\documentclass[a4paper,10pt]{scrartcl}

\usepackage[fleqn]{amsmath}
\usepackage{amssymb}
\usepackage{amsthm}
\usepackage{stmaryrd}
\usepackage{bbm}
\usepackage{hyperref}
\usepackage{enumitem}

\newcommand{\bE}{\ensuremath{\mathbb{E}}}

\newcommand{\bN}{\ensuremath{\mathbb{N}}}

\newcommand{\bP}{\ensuremath{\mathbb{P}}}

\newcommand{\bR}{\ensuremath{\mathbb{R}}}

\newcommand{\bZ}{\ensuremath{\mathbb{Z}}}

\newcommand{\ind}{\ensuremath{\mathbbm{1}}}

\newcommand{\cB}{\ensuremath{\mathcal{B}}}
\newcommand{\cC}{\ensuremath{\mathcal{C}}}

\newcommand{\fF}{\ensuremath{\mathfrak{F}}}

\newcommand{\bs}{\backslash}

\newcommand{\hqed}{\hfill\qed}

\newcommand{\abs}[1]{\left\vert \, #1 \, \right\vert}
\newcommand{\norm}[1]{\left\Vert \, #1 \, \right\Vert}

\newcommand{\normb}[1]{\interleave \, #1 \, \interleave}

\newcommand{\konv}[1]{\underset{#1}{\longrightarrow}}

\newcommand{\ddx}[1][1]{\ifnum#1=1 \frac{d}{dx} \else \frac{d^{#1}}{dx^{#1}} \fi}
\newcommand{\ddy}[1][1]{\ifnum#1=1 \frac{d}{dy} \else \frac{d^{#1}}{dy^{#1}} \fi}
\newcommand{\ddt}[1][1]{\ifnum#1=1 \frac{d}{dt} \else \frac{d^{#1}}{dt^{#1}} \fi}

\newcommand{\erwsymbol}{\mathbb{E}\,}

\newcommand{\erwc}[2]{\erwsymbol\left[#1\,\middle|\,#2\right]}

\newcommand{\suml}{\sum\limits}

\newcommand{\supl}{\sup\limits}
\newcommand{\infl}{\inf\limits}

\newtheorem{theorem}{Theorem}[section]
\newtheorem{lemma}[theorem]{Lemma}
\newtheorem{proposition}[theorem]{Proposition}
\newtheorem{corollary}[theorem]{Corollary}
\newtheorem{definition}[theorem]{Definition}
\newtheorem{notation}[theorem]{Notation}

\newenvironment{assumption}[1][]{\begin{trivlist}
\item[\hskip \labelsep {\bfseries Assumption #1}]\begin{em}}{\end{em}\end{trivlist}}

\renewenvironment{proof}[1][Proof]{\begin{trivlist}
\item[\hskip \labelsep {\bfseries #1}]}{\end{trivlist}}

\newenvironment{remark}[1][Remark]{\begin{trivlist}
\item[\hskip \labelsep {\bfseries #1}]}{\end{trivlist}}

\setenumerate[1]{label=\alph*),ref=\alph*}
\setenumerate[2]{label=\arabic*),ref=\theenumi.\arabic*}

\newcommand{\EP}{\widehat{\bP}^E}
\newcommand{\EE}{\widehat{\bE}^E}
\newcommand{\hP}{\widehat{\bP}}
\newcommand{\hE}{\widehat{\bE}}

\newcommand{\oLEP}{\overline{L}^{EP}}
\newcommand{\oLE}{\overline{L}^{E}}
\newcommand{\oL}{\overline{L}}
\newcommand{\oA}{\overline{A}}
\newcommand{\uA}{\underline{A}}

\title{Limit theorems for random walks in dynamic random environment}
\author{Frank Redig\thanks{University of Nijmegen, IMAPP, Heyendaalse weg 135, 6525 AJ Nijmegen, The Netherlands \newline
redig@math.leidenuniv.nl} \and Florian V\"ollering\thanks{Leiden University, Mathematical Institute, Niels Bohrweg 1, 2333 CA Leiden, The Netherlands \newline fvolleri@math.leidenuniv.nl}}

\begin{document}

\maketitle

\begin{abstract}
We study a general class of random walks driven by a uniquely ergodic Markovian environment. Under a coupling condition on the environment we obtain strong ergodicity properties and concentration inequalities for the environment as seen from the position of the walker, i.e the environment process. We also obtain ergodicity of the uniquely ergodic measure of the environment process as well as continuity as a function of the jump rates of the walker. 

As a consequence we obtain several limit theorems, such as law of large numbers, Einstein relation, central limit theorem and concentration properties for the position of the walker.
\end{abstract}
{\bf Keywords: } environment process, coupling, random walk, concentration estimates, backwards martingales.

\text{\bf AMS classification: 82C41(Primary) 60F17(Secondary)} 

\section{Introduction}
In recent days random walks in dynamic random environment have been studied
by several authors. Motivation comes among others from non-equilibrium statistical
mechanics -derivation of Fourier law- 
\cite{DOLGOPYAT:LIVERANI:08} and large deviation theory \cite{RASSOULAGHA:SAPPALAINEN:11}.
In principle random walk in dynamic random environment contains as a particular case
random walk in static random environment. However, mostly, in turning to dynamic
environments, authors concentrate more on environments with sufficient
mixing properties. In that case the fact that the environment
is dynamic helps to obtain self-averaging properties that ensure
standard limiting behavior of the walk, i.e., law of large numbers
and central limit theorem.

In the study of the limiting behavior of the walker, the environment
process, i.e., the environment as seen from the position
of the walker plays a crucial role. See also \cite{JOSEPH:RASSOULAGHA:10}, \cite{RASSOULAGHA:03} for the use of the environment process in related context. In a translation invariant setting the environment process is a Markov
process and its ergodic properties fully determine corresponding
ergodic properties of the walk, since the position of the walker
equals an additive function of the environment process plus a controllable martingale.

The main theme of this paper is precisely the study of the following
natural question: if the environment is uniquely ergodic, with
a sufficient speed of mixing, then the environment process
shares similar properties.
In several works (\cite{BOLDREGHINI:IGNATYUK:92}, \cite{BOLDREGHINI:MINLOS:07}, \cite{AVENA:DENHOLLANDER:REDIG:11}) this transfer
of ``good properties of the environment'' to ``similar properties
of the environment process'' is made via a perturbative argument,
and therefore holds only in a regime where the environment
and the walker are weakly coupled. An exception
is \cite{DOLGOPYAT:LIVERANI:09} in the case of an environment consisting
of independent Markov processes.

In this paper we consider the context of general Markovian
uniquely ergodic environments,
which are such that the semigroup contracts at a minimal
speed in norm of variation type.
Examples of such environments include interacting particle systems in ``the $M<\epsilon$
regime'' \cite{LIGGETT:05} and weakly interacting diffusion processes
on a compact manifold. Our conditions on the environment are formulated in the language of coupling.
More precisely, we impose that for the environment there exists
a coupling such that the distance between every pair of
initial configurations in this coupling decays
fast enough so that multiplied with $t^d$ it is still
integrable in time.
As a result we then obtain that for the environment process
there exists a coupling such that the distance between every pair of
initial configurations in this coupling decays
at a speed which is at least
integrable in time. In fact we show more, namely in going from the
environment to the environment process we essentially loose a factor $t^d$. E.g., 
if for the environment there is a coupling where
the distance decays exponentially, then the same
holds for the environment process (with possibly another
rate).

Once we have 
controllable coupling properties of the environment process,
we can draw strong conclusions for the position of the walker.
More precisely, we prove a law of large numbers with an asymptotic speed
that depends continuously on the rates, and a central limit
theorem with a controllable asymptotic covariance matrix.
We also prove recurrence in $d=1$ under condition
of zero speed and transience when the speed is non zero, as well as 
an Einstein relation.

Along the way, we develop a general formalism to derive concentration
inequalities for functions of a Markov process, based
on martingales adapted to quantities
like the position of the walker, which
is not exactly an additive functional of the environment process.
Concentration inequalities for the position of the walker are
crucial to obtain transience or recurrence.

The formalism to derive these concentration inequalities is in the
spirit of \cite{REDIG:VOLLERING:10} but based now on ``backwards'' martingales.
These ``backwards'' martingales are different from the classical
martingales associated to the generator, or arising in the
``martingale problem''. Using them
leads however to a controllable expression of
the variance of the walker, as well as for additive functionals
of the environment process.
We believe that the use of these backwards martingales are interesting per se,
and can also play
an important role in controlling the deviation from macroscopic
behavior (such as hydrodynamic limits)
in the context of interacting particle systems.

Our paper is organized as follows.
The model and necessary notation are introduced in Section \ref{section:model}. Section \ref{section:EP} is dedicated to lift properties of the environment to the environment process. Especially Theorem \ref{thm:EP-estimate} is of great importance and is used frequently in the later sections. Based on the results in Section \ref{section:EP} the Law of Large Numbers and an Einstein Relation are obtained in Section \ref{section:LLN}. A functional Central Limit Theorem is proven in Section \ref{section:CLT}, utilizing the martingale approach detailed in the Appendix. Section \ref{section:RWRE-concentration} further applies the general methods from the Appendix to the specific case of a random walk in dynamic random environment to obtain concentration estimates and therewith statements about recurrence and transience. Finally the Appendix is dedicated to the study of the non-time-homogeneous ``backwards'' martingales and the concentration inequalities they facilitate.

\section{The model}\label{section:model}
\subsection{Environment}
We are interested in studying a random walk $(X_t)_{t\geq0}$ on the lattice $\bZ^d$ which is driven by a second processes $(\eta_t)_{t\geq0}$ on $E^{\bZ^d}$, the (dynamic) environment. This can be interpreted as the random walk moving through the environment, and its transition rates being determined by the local environment around the random walk.

To become more precise, the \emph{environment} $(\eta_t)_{t\geq0}$ is a Feller Process on the state space $\Omega:=E^{\bZ^d}$, where $(E,\rho)$ is a compact Polish space with metric $\rho$ (examples in mind are $E=\{0,1\}$ or $E=[0,1]$). We assume that the distance $\rho$ on $E$ is bounded from above by 1. The generator of the Markov process is denoted by $L^E$ and its semigroup by $S_t^E$, both considered on the space of continuous functions $\cC(\Omega;\bR)$. We assume that the environment is translation invariant, i.e. \[ \bP^E_\eta(\theta_x\eta_t \in \cdot ) = \bP^E_{\theta_x\eta}(\eta_t \in \cdot) \] with $\theta_x$ denoting the shift operator $\theta_x\eta(y) = \eta(y-x)$ and $\bP^ E_\eta$ the path space measure of the process $(\eta_t)_{t\geq0}$ starting from $\eta$. Later on we will formulate precise conditions necessary to obtain our results.

\subsection{Lipschitz functions}
\newcommand{\Omegax}[1]{(\Omega\times\Omega)_{#1}}
For practical purposes, we introduce the set of pairs in $\Omega$ which differ at one specific site:
\[ \Omegax{x} := \left\{(\eta,\xi) \in \Omega^2 : \eta(x) \neq \xi(x) \text{ and } \eta(y)=\xi(y) \ \forall\,y\in\bZ^d\bs\{x\} \right\},\quad x\in\bZ^d.  \]

\begin{definition}
For any $f:\Omega\to\bR$, we denote by $\delta_f(x)$ the Lipschitz-constant of $f$ when only site $x$ is changed with respect to the distance $\rho$, i.e.
\[ \delta_f(x) := \supl_{(\eta,\xi)\in \Omegax{x}}\frac{f(\eta)-f(\xi)}{\rho(\eta(x),\xi(x))}. \]
We write 
\begin{align}
\normb{f}:=\suml_{x\in\bZ^d}\delta_f(x).	
\end{align}
\end{definition}
Note that $\normb{f}<\infty$ implies that $f$ is globally bounded and that the value of $f$ is uniformly weakly dependent on sites far away. A rougher semi-norm we also use is the oscillation (semi)-norm
\begin{align*}
	\norm{f}_{osc} := \sup_{\eta,\xi \in \Omega}\left(f(\eta)-f(\xi)\right).
\end{align*}
Generally, those two semi-norms are related by the inequality $\norm{f}_{osc} \leq \normb{f}$.

\subsection{The random walker and assumption on rates}\label{subsection:rates}
\newcommand{\dalpha}[1][z]{\delta_{\alpha(\cdot,#1)}}
The random walker $X_t$ is a process on $\bZ^d$, whose transition rates depend on the state of the environment as seen from the walker. More precisely, the rate to jump from site $x$ to site $x+z$ given that the environment is in state $\eta$ is $\alpha(\theta_{-x}\eta,z)$. 
We make two assumptions on the jump rates $\alpha$. First, we guarantee that the walker $X_t$ has first moments by assuming
\begin{align}\label{eq:rates-moment}
\norm{\alpha}_1 := \suml_{z\in\bZ^d}\norm{z}\supl_{\eta\in\Omega}\abs{\alpha(\eta,z)}<\infty.
\end{align}
More generally, as sometimes higher moments are necessary, we write
\[ \norm{\alpha}_p^p := \suml_{z\in\bZ^d}\norm{z}^p\supl_{\eta\in\Omega}\abs{\alpha(\eta,z)}, \quad p\geq1. \]
Second, we limit the sensitivity of the rates to small changes in the environment by assuming that
\begin{align}
 \normb{\alpha}:=\sum_{z\in\bZ^d} \normb{\alpha(\cdot,z)}<\infty.
\end{align}
In Section \ref{section:RWRE-concentration}, we impose the following stronger condition $\norm{\alpha}_1<\infty$:
\begin{align}
	\normb{\alpha}_1:=\sum_{z\in\bZ^d} \norm{z}\normb{\alpha(\cdot,z)}<\infty.
\end{align}

\subsection{Environment process}
While the random walker $X_t$ itself is not a Markov process due to the dependence on the environment, the pair $(\eta_t,X_t)$ is a Markov process with generator
\[ Lf(\eta,x) = L^Ef(\cdot,x)(\eta) + \suml_{z\in\bZ^d}\alpha(\theta_{-x}\eta,z)\left[f(\eta,x+z)-f(\eta,x)\right], \]
corresponding semigroup $S_t$ (considered on the space of functions continuous in $\eta\in\Omega$ and Lipschitz continuous in $x\in\bZ^d$) and path space measure $\bP_{\eta,x}$. 

The environment as seen from the walker is of crucial importance to understand the asymptotic behaviour of the walker itself. This process, $(\theta_{-X_t}\eta_t)_{t\geq0}$, is also called the environment process (this name is common in the literature, however in the context of this paper that name can easily be confused with the environment $\eta_t$), which is also a Markov process with generator
\[ L^{EP}f(\eta) = L^E f(\eta) + \suml_{z\in\bZ^d}\alpha(\eta,z)\left[f(\theta_{-z}\eta)-f(\eta)\right], \] 
corresponding semigroup $S_t^{EP}$ (on $\cC(\Omega)$) and path space measure $\bP^{EP}_{\eta}$. Notice that this process is well-defined only in the translation invariant context.
\section{Ergodicity of the environment process}\label{section:EP}
\subsection{Assumptions on the environment}
\newcommand{\sumsup}[1][x]{\suml_{x\in\bZ^d}\supl_{(\eta,\xi)\in\Omegax{#1}}}

In this section we will show how we can use a given Markovian coupling of the environment with a fast enough coupling speed to obtain a strong ergodicity properties about the environment process.
\begin{assumption}[1a.]
There exists a strong Markovian coupling $\EP$ of the environment (with corresponding expectation $\EE$) which satisfies
\[ \int_0^\infty t^d \supl_{\eta,\xi\in \Omega} \EE_{\eta,\xi}\rho(\eta_t^1(0),\eta_t^2(0))\,dt<\infty .\]
\end{assumption}
This assumption is already sufficient to obtain the law of large numbers for the position of the walker and unique ergodicity of the environment process,
but it does not give quite enough control on local fluctuations. The following stronger assumption remedies that.
\begin{assumption}[1b.]
There exists a strong Markovian coupling $\EP$ of the environment (with corresponding expectation $\EE$) which satisfies
\[ \int_0^\infty t^d \sumsup[0] \frac{\EE_{\eta,\xi}\rho(\eta_t^1(x),\eta_t^2(x))}{\rho(\eta(0),\xi(0))}\,dt<\infty .\]
\end{assumption}
\begin{remark}
The following \emph{telescoping} argument which will be used several times later on shows how Assumption 1b implies 1a:

Fix $\eta,\xi \in \Omega$. Let $(\zeta_n)_{n\geq0}$ be a sequence in $\Omega$ with $\zeta_0=\eta$, $\lim_{n\to\infty}\zeta_n =\xi$ and $|\{x\in\bZ^d : \zeta_n(x)\neq\zeta_{n-1}(x) \}|\leq1$ for all $n\geq1$. This sequence interpolates between the configurations $\eta$ and $\xi$ by single site changes. Using such a sequence, we can \emph{telescope over single site changes}, and using the triangle inequality for each $n$:
\[ \EE_{\eta,\xi}\rho(\eta_t^1(0),\eta_t^2(0)) \leq \sum_{n\in\bN} \EE_{\zeta_{n+1},\zeta_n}\rho(\eta_t^1(0),\eta_t^2(0)) .\]
We can assume that each site $x\in\bZ^d$ is at most once changed in the sequence $\zeta_n$. Therefore,
\[ \sum_{n\in\bN} \EE_{\zeta_{n+1},\zeta_n}\rho(\eta_t^1(0),\eta_t^2(0)) \leq \sumsup[x] \EE_{\eta,\xi}\rho(\eta_t^1(0),\eta_t^2(0)) .\]
Using the translation invariance of the environment and $\rho\leq1$ we obtain the estimate
\[ \EE_{\eta,\xi}\rho(\eta_t^1(0),\eta_t^2(0)) \leq \sumsup[0] \frac{\EE_{\eta,\xi}\rho(\eta_t^1(x),\eta_t^2(x))}{\rho(\eta(0),\xi(0))} .\]
Taking the supremum on the left hand site and integrating over time, we obtain that Assumption 1b indeed implies Assumption 1a.
\end{remark}

Examples which satisfy Assumption 1a and 1b include:
\begin{itemize}
\item interacting particle systems in the so-called $M<\epsilon$;
\item weakly interacting diffusions on a compact manifold;
\item a system of ODEs which converges uniformly to its unique stationary configuration at sufficient speed.
\end{itemize}
Note that the third example is a deterministic environment and as such any form of condition which measures this convergence purely in probabilistic terms like the total variation distance is bound to fail. 
\subsection{Statement of the main theorem}
The main result of this section is the following theorem, which tells us how the coupling property of the environment lifts to the environment process.
\begin{theorem}\label{thm:EP-estimate}
Let $f:\Omega\rightarrow \bR$ with $\normb{f}<\infty$. 
\begin{enumerate}
\item Under Assumption 1a, there exists a constant $C_a>0$ so that
\[ \supl_{\eta,\xi\in \Omega}\int_0^\infty \abs{S^{EP}_t f(\eta)-S^{EP}_t f(\xi)}\,dt\leq C_a\normb{f}. \]
\item Under Assumption 1b, there exists a constant $C_b>0$ so that
\[ \sumsup\int_0^\infty \frac{\abs{S^{EP}_t f(\eta)-S^{EP}_t f(\xi)}}{\rho(\eta(x),\xi(x))}\,dt\leq C_b\normb{f}. \]
\end{enumerate}
\end{theorem}
This theorem encapsulates all the technical details and difficulties to obtain results, and subsection \ref{subsection:proof} is dedicated to its proof. In Section \ref{section:convergence-speed} we generalize this result to give more information about decay in time. Here we continue with results we can obtain using Theorem \ref{thm:EP-estimate}. Most results about the environment process just use part a) of the theorem, part b) is needed to obtain the CLT in Section \ref{section:CLT} or more sophisticated deviation estimates in Section \ref{section:RWRE-concentration}.

\subsection{Existence of a unique ergodic measure and continuity in the rates}
First of, the environment process, i.e. the environment as seen from the walker, is ergodic.
\begin{lemma}\label{lemma:ergodicity}
Under Assumption 1a the environment process has a unique ergodic probability measure $\mu^{EP}$.
\end{lemma}
\begin{proof}
As $E$ is compact, so is $\Omega$, and therefore the space of stationary measures is non-empty. So we must just prove uniqueness.

Assume $\mu,\nu$ are both stationary measures. Choose an arbitrary $f : \Omega\rightarrow \bR$ with $\normb{f}<\infty$. By Theorem \ref{thm:EP-estimate},a,
\begin{align*} 
T\abs{\mu(f)-\nu(f)} &\leq \int\int\int_0^T \abs{S^{EP}_t f (\eta)-S^{EP}_tf(\xi)}\,dt\,\mu(d\eta)\,\nu(d\xi) \\
& \leq \sup_{\eta,\xi\in \Omega} \int_0^\infty\abs{ S^{EP}_t f (\eta)-S^{EP}_t f(\xi)}\,dt<\infty. 
\end{align*}
As $T$ is arbitrary, $\mu(f)=\nu(f)$. As functions $f$ with $\normb{f}<\infty$ are dense in $\cC(\Omega)$, there is at most one stationary probability measure.
\qed
\end{proof}
It is of interest not only to know that the environment process has a unique ergodic measure $\mu^{EP}$, but also to know how this measure depends on the rates $\alpha$. 
\begin{theorem}\label{thm:mu-continuous}
Under Assumption 1a, the unique ergodic measure $\mu^{EP}_{\alpha}$ depends continuously on the rates $\alpha$. For two transition rate functions $\alpha, \alpha'$, we have the following estimate:
\[ \abs{\mu^{EP}_{\alpha}(f)-\mu^{EP}_{\alpha'}(f)} 
\leq \frac{C(\alpha)}{p(\alpha)}\norm{\alpha-\alpha'}_0 \normb{f}, \]
i.e. 
\[ (\alpha,f) \mapsto \mu^{EP}_\alpha(f) \]
is continuous in $\norm{\cdot}_0\times \normb{\cdot}$.
The functions $C(\alpha), p(\alpha)$ satisfy $C(\alpha)>0, p(\alpha)\in\,]0,1[$. In the case that the rates $\alpha$ do not depend on the environment, i.e. $\alpha(\eta,z)=\alpha(z)$, they are given by $p(\alpha)=1$,
\[ C(\alpha) = \int_0^\infty \sup_{\eta,\xi\in \Omega} \EE_{\eta,\xi}\rho(\eta^1_t(0),\eta^2_t(0))\,dt. \]
\end{theorem}
As the proof is a variation of the proof of Theorem \ref{thm:EP-estimate}, it is delayed to the end of Section \ref{subsection:proof}.

\subsection{Speed of convergence to equilibrium in the environment process}\label{section:convergence-speed}
We already know that under Assumption 1a the environment process has a unique ergodic distribution. However, we do not know at what speed the process converges to its unique stationary measure. Given the speed of convergence for the environment it is natural to believe that the environment process inherits that speed with some form of slowdown due to the additional self-interaction which is induced from the random walk. For example, if the original speed of convergence were exponential with a constant $\lambda$, then the environment process would also converge exponentially fast, but with a worse constant $\widetilde\lambda$. This is indeed the case.

\newcommand{\phiK}[1][t]{\phi\left(\frac{#1}{K}\right)}
\begin{theorem}\label{thm:EP-phi-estimate}
Let $\phi:[0,\infty[\ \to\bR$ be a monotone increasing and continuous function satisfying $\phi(0)=1$ and $\phi(s+t)\leq \phi(s)\phi(t)$. 
\begin{enumerate}
\item Suppose there exists a strong Markovian coupling $\EP$ of the environment (with corresponding expectation $\EE$) which satisfies
\[ \int_0^\infty \phi(t) t^d \supl_{\eta,\xi\in \Omega} \EE_{\eta,\xi}\rho(\eta_t^1(0),\eta_t^2(0))\,dt<\infty .\]
Then there exists a constant $K_0>0$ and a decreasing function $C_a :\ ]K_0,\infty[\ \to [0,\infty[$ so that for any $K>K_0$ and any $f:\Omega\rightarrow \bR$ with $\normb{f}<\infty$,
\[ \supl_{\eta,\xi\in \Omega}\int_0^\infty \phiK \abs{S^{EP}_t f(\eta)-S^{EP}_t f(\xi)}\,dt\leq C_a(K)\normb{f}. \]
\item Suppose there exists a strong Markovian coupling $\EP$ of the environment (with corresponding expectation $\EE$) which satisfies
\[ \int_0^\infty \phi(t) t^d \sumsup[0] \frac{\EE_{\eta,\xi}\rho(\eta_t^1(x),\eta_t^2(x))}{\rho(\eta(0),\xi(0))}\,dt<\infty .\]
Then there exists a constant $K_0>0$ and a decreasing function $C_b :\ ]K_0,\infty[\ \to [0,\infty[$ so that for any $K>K_0$ and any $f:\Omega\rightarrow \bR$ with $\normb{f}<\infty$,
\[ \sumsup\int_0^\infty \phiK \frac{\abs{S^{EP}_t f(\eta)-S^{EP}_t f(\xi)}}{\rho(\eta(x),\xi(x))}\,dt\leq C_b(K)\normb{f}. \]
\end{enumerate}
\end{theorem}
Canonical choices for $\phi$ are $\phi(t)=\exp(\lambda t)$ or $\phi(t)= (1+t)^\lambda$. In the first case, exponential decay of order $\lambda$ of the environment is lifted to exponential decay of order $\lambda/K$ for any $K>K_0$ in the environment process. In the second case, polynomial decay of order $\lambda$ becomes polynomial decay of order $\lambda - d - \epsilon$ for any $0<\epsilon\leq\lambda-d$.

\subsection{Proofs of Theorem \ref{thm:EP-estimate}, \ref{thm:mu-continuous} and \ref{thm:EP-phi-estimate}}\label{subsection:proof}
In this section we always assume that Assumption 1a holds. 

We start with an outline of the idea of the proofs. We have a coupling of the environments $(\eta^1_t,\eta^2_t)$, which we extend to include two random walkers $(X^1_t,X^2_t)$, driven by their corresponding environment. We maximize the probability of both walkers performing the same jumps. Then Assumption 1a is sufficient to obtain a positive probability of both walkers staying together forever. If the walkers stay together, one just has to account for the difference in environments, but not the walkers as well. When the walkers split, the translation invariance allows for everything to shifted that both walkers are back at the origin, and one can try again. After a geometric number of trials it is then guaranteed that the walkers stay together.

\begin{proposition}[Coupling construction]\label{prop:coupling-construction}
Given the coupling $\EP_{\eta,\xi}$ of the environments, we extend it to a coupling $\hP_{\eta,x;\xi,y}$. This coupling has the following properties:
\begin{enumerate}
\item (Marginals)The coupling supports two environments and corresponding random walkers:
\begin{enumerate}
\item $\hP_{\eta,x;\xi,y}((\eta^1_t,X_t^1)\in \cdot) = \bP_{\eta,x}((\eta_t,X_t) \in \cdot)$;
\item $\hP_{\eta,x;\xi,y}((\eta^2_t,X_t^2)\in \cdot) = \bP_{\xi,y}((\eta_t,X_t) \in \cdot)$;
\end{enumerate}
\item (Extension of $\EP_{\eta,\xi}$) The environments behave as under $\EP$:
\[ \hP_{\eta,x;\xi,y}((\eta^1_t,\eta_t^2)\in \cdot) = \EP_{\eta,\xi}((\eta_t^1,\eta_t^2) \in \cdot); \]
\item\label{enum:decoupling-rate} (Coupling of the walkers) $X_t^1$ and $X_t^2$ perform identical jumps as much as possible, the rate of performing a different jump is $\suml_{z\in\bZ^d}\abs{\alpha(\theta_{-X_t^1}\eta_t^1,z)-\alpha(\theta_{-X_t^2}\eta_t^2,z)}$;
\item (Minimal and maximal walkers) In addition to the environments $\eta_t^1$ and $\eta_t^2$ and random walkers $X_t^1$ and $X_t^2$, the coupling supports minimal and maximal walkers $Y_t^+$, $Y_t^-$ as well. These two walkers have the following properties:
\begin{enumerate}
\item $Y_t^- \leq X_t^1-x, X_t^2-y \leq Y_t^+ \quad \hP_{\eta,x;\xi,y}-a.s.$ (in dimension $d>1$, this is to be interpreted coordinate-wise);
\item $Y_t^+,Y_t^-$ are independent of $\eta_t^1,\eta_t^2$;
\item $\hE_{\eta,x;\xi,y}Y_t^+ = t\gamma^+$ for some $\gamma^+\in\bR^d$;
\item $\hE_{\eta,x;\xi,y}Y_t^- = t\gamma^-$ for some $\gamma^-\in\bR^d$.
\end{enumerate}
\end{enumerate}
\end{proposition}
\begin{proof}
The construction of this coupling $\hP_{\eta,x;\xi,y}$ can be done in the following way: The environments behave according to $\EP_{\eta,\xi}$, and $X_0^1=x$, $X_0^2=y$, $Y_0^+=Y_0^-=0$. For each $z\in\bZ^d$ there is an independent Poisson clock with rate $\lambda_z:=\sup_\eta \alpha(\eta,z)$. Whenever such a clock rings, draw an independent $U$ from the uniform distribution in $[0,1]$. If $U<\alpha(\theta_{-X_t^1}\eta_t^1,z)/\lambda_z$, $X_t^1$ performs a jump by $z$, analogue for $X_t^2$. The upper and lower walkers $Y_t^+$ and $Y_t^-$ always jump on these clocks, however they jump by $\max(z,0)$ or $\min(z,0)$ respectively. $N_t$ increases if 

The properties of the coupling arise directly from the construction, except the last two. Those are a consequence from the first moment condition $\norm{\alpha}_1<\infty$.\qed
\end{proof}
Now we show how suitable estimates on the coupling speed of the environment translate to properties of the extended coupling.
\begin{lemma}\label{lemma:supnorm-coupling-estimate}
\[ \hE_{\eta,x;\xi,y}\rho(\eta_t^1(X_t^1), \eta_t^2(X_t^1))\,dt<(\norm{\gamma^{+} - \gamma^{-}}_\infty t + 1)^d\supl_{\eta,\xi\in\Omega}\EE_{\eta,\xi}\rho(\eta_t^1(0),\eta_t^2(0)). \]
\end{lemma}
\begin{proof}
Denote with $R_t \subset \bZ^d$ the set of sites $x$ with $Y_t^-\leq x \leq Y_t^+$(coordinate-wise). Then
\begin{align*}
&\supl_{\eta,\xi,x,y} \hE_{\eta,x;\xi,y}\rho(\eta_t^1(X_t^1), \eta_t^2(X_t^1))\\
&\quad\leq \supl_{\eta,\xi,x,y} \hE_{\eta,x;\xi,y}\suml_{z\in R_t}\rho(\eta_t^1(x+z), \eta_t^2(x+z))\\
&\quad\leq \hE \left[\suml_{z\in R_t}1\right]\supl_{\eta,\xi,z}\EE_{\eta,\xi}\rho(\eta_t^1(z), \eta_t^2(z)) \\
&\quad\leq (\norm{\gamma^{+} - \gamma^{-}}_\infty t + 1)^d \supl_{\eta,\xi\in\Omega}\EE_{\eta,\xi}\rho(\eta_t^1(0), \eta_t^2(0)).
\end{align*}
\qed
\end{proof}

\begin{lemma}\label{lemma:decoupling-estimate}
Denote by $\tau:=\inf\{t\geq0:X^1_t\neq X_t^2\}$ the first time the two walkers are not at the same position. Under Assumption 1a,
\[ \infl_{\eta,\xi\in\Omega} \hP_{\eta,0;\xi,0}(\tau=\infty) >0 ,\]
i.e., the walkers $X^1$ and $X^2$ never decouple with strictly positive probability.
\end{lemma}
\begin{proof}
Both walkers start in the origin, therefore $\tau>0$. The probability $\bP(\exp((\lambda_t))>T)$ that an exponential random variable with time dependent rate $\lambda_t$ is bigger than $T$ is given by $\exp({-\int_0^T\lambda_t\,dt})$. As the rate of decoupling is given by Proposition \ref{prop:coupling-construction},c), we obtain
\begin{align} \hP_{\eta,0;\xi,0}(\tau > T) &= \hE_{\eta,0;\xi,0} \exp\left({-\int_0^T \suml_{z\in\bZ^d}\abs{\alpha(\theta_{-X^1_t}\eta^1_t,z)- \alpha(\theta_{-X^1_t}\eta^2_t,z)}\,dt}\right) \label{eq:dynamic-rate-poisson}\\
&\geq \exp\left(-\hE_{\eta,0;\xi,0}{\int_0^T \suml_{z\in\bZ^d}\abs{\alpha(\theta_{-X^1_t}\eta^1_t,z)- \alpha(\theta_{-X^1_t}\eta^2_t,z)}\,dt}\right) . \nonumber
\end{align}
By telescoping over single site changes,
\begin{align*}
&\hE_{\eta,0;\xi,0}\suml_{z\in\bZ^d}\abs{\alpha(\theta_{-X^1_t}\eta^1_t,z)- \alpha(\theta_{-X^1_t}\eta^2_t,z)} \\
&\quad \leq \hE_{\eta,0;\xi,0}\suml_{z\in\bZ^d}\suml_{x\in\bZ^d}\rho(\eta^1_t(X^1_t+x),\eta^2_t(X^1_t+x))\dalpha(x)\\
&\quad \leq \supl_{x\in\bZ^d}\hE_{\eta,0;\xi,0}\rho(\eta^1_t(X^1_t+x),\eta^2_t(X^1_t+x))\normb{\alpha}\\
&\quad \leq \normb{\alpha}(\norm{\gamma^{+} - \gamma^{-}}_\infty t + 1)^d\supl_{\eta,\xi\in\Omega}\EE_{\eta,\xi}\rho(\eta_t^1(0),\eta_t^2(0)),
\end{align*}
where the last line follows from Lemma \ref{lemma:supnorm-coupling-estimate}. With this estimate, we obtain
\begin{align*}
\hP_{\eta,\xi}(\tau =\infty ) &\geq \exp\left(- \normb{\alpha}\int_0^\infty(\norm{\gamma^{+} - \gamma^{-}}_\infty t + 1)^d\supl_{\eta,\xi\in\Omega}\EE_{\eta,\xi}\rho(\eta_t^1(0),\eta_t^2(0))\,dt \right)
\\
&>0 \quad \text{uniformly in $\eta, \xi$}.
\end{align*}
\qed
\end{proof}
\begin{proof}[Proof of Theorem \ref{thm:EP-estimate}, part a)]
The idea of the proof is to use the coupling of Proposition \ref{prop:coupling-construction}: We wait until the walkers $X_t^1$ and $X_t^2$, which are initially at the same position, decouple, and then restart everything and try again. By Lemma \ref{lemma:decoupling-estimate} there is a positive probability of never decoupling, so this scheme is successful. Using the time of decoupling $\tau$ (as in Lemma \ref{lemma:decoupling-estimate}), 
\begin{align}
&\int_0^T \abs{\hE_{\eta,0;\xi,0}\ind_{t\geq\tau}\left(f(\theta_{-X^1_t}\eta^1_t)-f(\theta_{-X^2_t}\eta^2_t)\right) }\,dt \nonumber\\
&\quad= \int_0^T \abs{\hE_{\eta0;\xi,0}\ind_{t\geq\tau}\erwc{f(\theta_{-X^1_t}\eta^1_t)-f(\theta_{-X^2_t}\eta^2_t)}{\fF_\tau}}\,dt \nonumber\\
&\quad \leq \int_0^T \hE_{\eta,0;\xi,0}\ind_{t\geq\tau}\abs{ S^{EP}_{t-\tau}f(\theta_{-X^1_\tau}\eta^1_\tau)-S^{EP}_{t-\tau}f(\theta_{-X^2_\tau}\eta^2_\tau) }\,dt \nonumber\\
&\quad = \hE_{\eta,0;\xi,0} \int_0^{(T-\tau)\vee0} \abs{ S^{EP}_{t}f(\theta_{-X^1_\tau}\eta^1_\tau)-S^{EP}_{t}f(\theta_{-X^2_\tau}\eta^2_\tau) }\,dt \label{eq:phi-1}\\
&\quad \leq \hP_{\eta,0;\xi,0}(\tau<\infty) \supl_{\eta,\xi\in\Omega}\int_0^{T} \abs{ S^{EP}_{t}f(\eta)-S^{EP}_{t}f(\xi) }\,dt. \label{eq:phi-2}
\end{align}
And therefore
\begin{align}
&\int_0^T \abs{S^{EP}_t f(\eta)-S^{EP}_t f(\xi)}\,dt \nonumber\\
&\quad= \int_0^T \abs{\hE_{\eta,0;\xi,0}f(\theta_{-X^1_t}\eta^1_t)-f(\theta_{-X^2_t}\eta^2_t) }\,dt \nonumber\\
&\quad\leq \int_0^T \hE_{\eta,0;\xi,0}\ind_{t<\tau}\abs{f(\theta_{-X^1_t}\eta^1_t)-f(\theta_{-X^1_t}\eta^2_t) }\,dt \nonumber\\
&\qquad+ \hP_{\eta,0;\xi,0}(\tau<\infty) \supl_{\eta,\xi\in\Omega}\int_0^{T} \abs{ S^{EP}_{t}f(\eta)-S^{EP}_{t}f(\xi) }\,dt \label{eq:phi-3}\\
&\quad\leq \int_0^\infty \hE_{\eta,0;\xi,0}\abs{f(\theta_{-X^1_t}\eta^1_t)-f(\theta_{-X^1_t}\eta^2_t) }\,dt \nonumber\\
&\qquad+ \hP_{\eta,0;\xi,0}(\tau<\infty)\supl_{\eta,\xi\in\Omega}\int_0^T \abs{ S^{EP}_{t}f(\eta)-S^{EP}_{t}f(\xi) }\,dt, \label{eq:phi-4}
\end{align}
which gives us the upper bound
\begin{align}
 &\supl_{\eta,\xi\in\Omega}\int_0^\infty \abs{S^{EP}_tf(\eta)-S^{EP}_tf(\xi)}\,dt \nonumber\\
 &\quad\leq \left(\infl_{\eta,\xi\in\Omega}\hP_{\eta,0;\xi,0}(\tau=\infty)\right)^{-1}\supl_{\eta,\xi\in\Omega}\int_0^\infty \hE_{\eta,0;\xi,0}\abs{f(\theta_{-X^1_t}\eta^1_t)-f(\theta_{-X^1_t}\eta^2_t) }\,dt . \label{eq:phi-5}
\end{align}
To show that the last integral is finite, we telescope over single site changes, and get
\begin{align*}
&\int_0^\infty \hE_{\eta,0;\xi,0}\abs{f(\theta_{-X^1_t}\eta^1_t)-f(\theta_{-X^1_t}\eta^2_t) }\,dt \\
&\quad\leq \int_0^\infty \hE_{\eta,0;\xi,0}\suml_{x\in\bZ^d}\rho(\eta^1_t(x+X^1_t),\eta^2_t(x+X^1_t))\delta_f(x)\,dt\\
&\quad\leq \normb{f}\supl_{\eta,\xi,x}\int_0^\infty \hE_{\eta,0;\xi,0}\rho(\eta^1_t(x+X_t^1),\eta^2_t(x+X_t^1))\,dt,
\end{align*}
which is finite by Lemma \ref{lemma:supnorm-coupling-estimate} and Assumption 1a). Choosing
\begin{align}
C_a = \left(\infl_{\eta,\xi\in\Omega}\hP_{\eta,0;\xi,0}(\tau=\infty)\right)^{-1}\supl_{\eta,\xi,x}\int_0^\infty \hE_{\eta,0;\xi,0}\rho(\eta^1_t(x+X_t^1),\eta^2_t(x+X_t^1))\,dt \label{eq:phi-6}	
\end{align}
completes the proof.
\qed
\end{proof}
To prove part b) of the theorem, we need the following analogue to Lemma \ref{lemma:decoupling-estimate} using Assumption 1b.
\begin{lemma}\label{lemma:triplenorm-coupling-estimate}
Under Assumption 1b, for every site-weight function $w:\bZ^d \to [0,\infty[$ with $\norm{w}_1 := \suml_x w(x)<\infty$,
we have
\[ \sumsup\int_0^\infty \suml_{y\in\bZ^d} w(y)\frac{\hE_{\eta,\xi}\rho(\eta_t^1(y+X_t^1), \eta_t^2(y+X_t^1))}{\rho(\eta(x),\xi(x))}\,dt\leq const \cdot \norm{w}_1 . \]
\end{lemma}
\begin{proof}
Denote with $R_t \subset \bZ^d$ the set of sites whose $j$th coordinate lies between $Y_t^{j,-}$ and $Y_t^{j,+}$. Then,
\begin{align*}
& \suml_{y\in\bZ^d} w(y)\hE_{\eta,\xi}\rho(\eta_t^1(y+X_t^1), \eta_t^2(y+X_t^1)) \\
&\quad= \suml_{y\in\bZ^d} \hE_{\eta,\xi}w(y-X_t^1)\rho(\eta_t^1(y), \eta_t^2(y)) \\
&\quad\leq \suml_{y\in\bZ^d} \hE_{\eta,\xi}\suml_{z\in R_t}w(y-z) \rho(\eta_t^1(y), \eta_t^2(y))\\
&\quad = \suml_{y\in\bZ^d} \bE\left[\suml_{z\in R_t}w(y-z)\right]\EE_{\eta,\xi} \rho(\eta_t^1(y), \eta_t^2(y)) 
\end{align*}
by independence of $R_t$ and $(\eta^1_t,\eta^2_t)$. Therewith,
\begin{align*}
&\sumsup\int_0^\infty \suml_{y\in\bZ^d} w(y)\frac{\hE_{\eta,\xi}\rho(\eta_t^1(y+X_t^1), \eta_t^2(y+X_t^1))}{\rho(\eta(x),\xi(x))}\,dt \\
&\quad\leq \int_0^\infty \suml_{y\in\bZ^d} \bE\left[\suml_{z\in R_t}w(y-z)\right] \sumsup \EE_{\eta,\xi} \frac{\rho(\eta_t^1(y), \eta_t^2(y))}{\rho(\eta(x),\xi(x))}\,dt .
\end{align*}
Note that by translation invariance the right part is equal to 
\[ \sumsup[0] \EE_{\eta,\xi} \frac{\rho(\eta_t^1(x), \eta_t^2(x))}{\rho(\eta(0),\xi(0))} \] 
and by construction of $R_t$ and Proposition \ref{prop:coupling-construction}.d , 
\begin{align*}
 \suml_{y\in\bZ^d} \bE\left[\suml_{z\in R_t}w(y-z)\right] 
 &= \bE \left[\suml_{z\in R_t} 1\right]\norm{w}_1 
 = \prod_{j=1}^d(\gamma^{j,+}t - \gamma^{j,-}t+1) \norm{w}_1 \\
 &\leq c_w (t^d+1)
\end{align*}
for some suitable $c_w>0$.
Therefore Assumption 1b completes the proof.
\qed
\end{proof}
\begin{proof}[Proof of Theorem \ref{thm:EP-estimate}, part b)]
Let $\tau:=\inf\{t\geq0:X^1_t\neq X^2_t\}$. Then we split the integration at $\tau$:
\begin{align*}
&\sumsup\int_0^\infty \frac{\abs{S_t^{EP} f(\eta)-S_t^{EP} f(\xi)}}{\rho(\eta(x),\xi(x))}\,dt \\
&\quad\leq \sumsup \int_0^\infty \frac{\abs{\hE_{\eta,\xi}\ind_{\tau>t}\left(f(\theta_{-X_t^1}\eta_t^1)-f(\theta_{-X_t^1}\eta_t^2) \right)}}{\rho(\eta(x),\xi(x))} \,dt\\
&\qquad +\sumsup \int_0^\infty\frac{\abs{\hE_{\eta,\xi}\ind_{\tau\leq t} \left(f(\theta_{-X_t^1}\eta_t^1)-f(\theta_{-X_t^2}\eta_t^2)\right)}}{\rho(\eta(x),\xi(x))} \,dt 
\end{align*}
We estimate the first term by moving the expectation out of the absolute value and forgetting the restriction to $\tau>t$:
\[ \sumsup \int_0^\infty \suml_{y\in\bZ^d}\delta_f(y)\frac{\hE_{\eta,\xi}\rho(\eta_t^1(y+X_t^1), \eta_t^2(y+X_t^1))}{\rho(\eta(x),\xi(x))} \,dt. \]
By Lemma \ref{lemma:triplenorm-coupling-estimate} with $w = \delta_f$, this is bounded by some constant times $\normb{f}$.
For the second term we start by using the Markov property:
\begin{align}
& \int_0^\infty\abs{\hE_{\eta,\xi}\ind_{\tau\leq t} \left(f(\theta_{-X_t^1}\eta_t^1)-f(\theta_{-X_t^2}\eta_t^2)\right)}\,dt \nonumber\\
&\quad=  \int_0^\infty\abs{\hE_{\eta,\xi}\ind_{\tau\leq t} \left(S_{t-\tau}^{EP}f(\theta_{-X_\tau^1}\eta_\tau^1)-S_{t-\tau}^{EP}f(\theta_{-X_\tau^2}\eta_\tau^2)\right)}\,dt \nonumber\\
&\quad\leq  \hE_{\eta,\xi}\ind_{\tau<\infty}\int_\tau^\infty\abs{ \left(S_{t-\tau}^{EP}f(\theta_{-X_\tau^1}\eta_\tau^1)-S_{t-\tau}^{EP}f(\theta_{-X_\tau^2}\eta_\tau^2)\right)}\,dt \label{eq:phi-7}\\
&\quad\leq 
\hP_{\eta,\xi}(\tau<\infty)\supl_{\eta,\xi \in \Omega}\int_0^\infty\abs{S_t^{EP} f(\eta)-S_t^{EP} f(\xi)}\,dt .\label{eq:phi-8}
\end{align}
By part a) of Theorem \ref{thm:EP-estimate} the integral part is uniformly bounded by $C_a \normb{f}$. So what remains to complete the proof is to show that
\begin{align}
 \sumsup \frac{\hP_{\eta,\xi}(\tau<\infty)}{\rho(\eta(x),\xi(x))}< \infty. \label{eq:phi-9}	
\end{align}

To do so we first use the same idea as in the proof of Lemma \ref{lemma:decoupling-estimate} to obtain
\begin{align}
&\hP_{\eta,\xi}(\tau<\infty)  \nonumber\\
&\quad=	1-\exp\left(- \int_0^\infty \hE_{\eta,\xi} \suml_{z\in\bZ^d}\abs{\alpha(\theta_{-X_t^1}\eta_t^1,z) - \alpha(\theta_{-X_t^1}\eta_t^2,z)}\,dt\right) \nonumber\\
&\quad\leq \int_0^\infty \hE_{\eta,\xi} \suml_{z\in\bZ^d}\abs{\alpha(\theta_{-X_t^1}\eta_t^1,z) - \alpha(\theta_{-X_t^1}\eta_t^2,z)}\,dt \nonumber\\
&\quad\leq	\int_0^\infty \suml_{y\in\bZ^d} w_\alpha(y) \hE_{\eta,\xi}\rho(\eta_t^1(y+X_t^1), \eta_t^2(y+X_t^1)) \,dt \nonumber
\end{align}
with 
\[ w_\alpha(x) := \supl_{(\eta,\xi)\in\Omegax{x}} \suml_{z\in\bZ^d}\abs{\alpha(\eta,z) - \alpha(\xi,z)} \]
and $\sum_{x\in\bZ^d}w_\alpha(x)<\infty$.
So we get
\begin{align*}
&\sumsup \frac{\hP_{\eta,\xi}(\tau<\infty)}{\rho(\eta(x),\xi(x))}  \\
&\quad\leq	\sumsup\int_0^\infty \suml_{y\in\bZ^d} w_\alpha(y) \frac{\hE_{\eta,\xi}\rho(\eta_t^1(y+X_t^1), \eta_t^2(y+X_t^1))}{\rho(\eta(x),\xi(x))} \,dt, \\
\end{align*}
and Lemma \ref{lemma:triplenorm-coupling-estimate} completes the proof, where $C_b$ is the combination of the various factors in front of $\normb{f}$.\qed
\end{proof}

\begin{proof}[Proof of Theorem \ref{thm:mu-continuous}]
Let $\alpha, \alpha'$ be two different transition rates. The goal is to show that
\[ \abs{\mu^{EP}_{\alpha}(f)-\mu^{EP}_{\alpha'}(f)} \leq C \normb{f} \]
for all $f:\Omega\to\bR$ with $\normb{f}<\infty$.

The idea is now to use a coupling $\hP$ similar to the one in Proposition \ref{prop:coupling-construction}. 
The coupling contains as objects two copies of the environment, $\eta^1$ and $\eta^2$, and three random walks, $X^1,X^{12}$ and $X^2$. The random walk $X^1$ moves on the environment $\eta^1$ with rates $\alpha$, and correspondingly the random walk $X^2$ moves on $\eta^2$ with rates $\alpha'$. The mixed walker $X^{12}$ moves on the environment $\eta^2$ as well, but according to the rates $\alpha$. The walkers $X^1,X^2$ will perform the same jumps as $X^{12}$ with maximal probability. The environments we couple utilizing the Markovian coupling provided by Assumption 1a.
We only consider the case where all three walkers start at the origin.
We denote by $S_t^{EP,1}$, $S_t^{EP,2}$ the semigroups of the environment process which correspond to the rates $\alpha$ and $\alpha'$.
Let $\tau:=\inf\{t\geq0 : X^1_t \neq X^{12}_t \text{ or } X^{12}_t \neq X^{2}_t\}$.

\begin{align*}
&S_t^{EP,1}f(\eta) - S_t^{EP,2}f(\xi) \\
&\quad = \hE_{\eta,\xi} \left( f(\theta_{-X_t^1}\eta^1_t) - f(\theta_{-X_t^2}\eta^2_t) \right)\\
&\quad = \hE_{\eta,\xi} \ind_{\tau>t}\left( f(\theta_{-X_t^1}\eta^1_t) - f(\theta_{-X_t^1}\eta^2_t) \right)\\
&\qquad + \hE_{\eta,\xi} \ind_{\tau\leq t}\left( f(\theta_{-X_t^1}\eta^1_t) - f(\theta_{-X_t^2}\eta^2_t) \right) \\
&\quad = \hE_{\eta,\xi} \ind_{\tau>t}\left( f(\theta_{-X_t^1}\eta^1_t) - f(\theta_{-X_t^1}\eta^2_t) \right)\\
&\qquad + \hE_{\eta,\xi} \ind_{\tau\leq t}\left( S_{t-\tau}^{EP,1}f(\theta_{-X_\tau^1}\eta^1_\tau) - S_{t-\tau}^{EP,2}f(\theta_{-X_\tau^2}\eta^2_\tau) \right). \\
\end{align*}

Therefore, 
\begin{align}
\Psi(T)&:=\sup_{0\leq T'\leq T}\sup_{\eta,\xi\in \Omega} \int_0^{T'} S_t^{EP,1}f(\eta) - S_t^{EP,2}f(\xi) \,dt \nonumber\\
&\leq \sup_{0\leq T'\leq T}\sup_{\eta,\xi\in \Omega} \int_0^{T'} \hE_{\eta,\xi}  \ind_{\tau>t}\left( f(\theta_{-X_t^1}\eta^1_t) - f(\theta_{-X_t^1}\eta^2_t) \right)\nonumber\\
&\qquad + \hE_{\eta,\xi} \ind_{\tau\leq t}\sup_{\eta,\xi\in \Omega}\left( S_{t-\tau}^{EP,1}f(\eta) - S_{t-\tau}^{EP,2}f(\xi) \right)\,dt \nonumber\\
& \leq \sup_{0\leq T'\leq T}\sup_{\eta,\xi\in \Omega} \left(\hE_{\eta,\xi} \int_0^\tau f(\theta_{-X_t^1}\eta^1_t) - f(\theta_{-X_t^1}\eta^2_t) \,dt + \ind_{\tau\leq T'}\Psi(T'-\tau)\right) \nonumber\\
& \leq \sup_{\eta,\xi\in \Omega} \left(\hE_{\eta,\xi} \int_0^\infty f(\theta_{-X_t^1}\eta^1_t) - f(\theta_{-X_t^1}\eta^2_t) \,dt + \ind_{\tau\leq T}\Psi(T-\tau)\right). \label{eq:PsiT}
\end{align}
We will now exploit this recursive bound on $\Psi$.
\begin{lemma}
Let $\tau_1:= \inf\{t\geq0 : X^1_t \neq X^{12}_t\}$ and $\tau_2:= \inf\{t\geq0 : X^{12}_t \neq X^{2}_t\}$.
Set 
\begin{align*}
\beta &:=\sum_{z\in\bZ^d} \sup_{\eta\in\bZ^d}\abs{\alpha(\eta,z)-\alpha'(\eta,z)},\\
p(\alpha)&:= \inf_{\eta,\xi\in\Omega} \hP_{\eta,\xi}(\tau_1=\infty), \\
C(\alpha) &:= \int_0^\infty \left( \norm{\gamma^+(\alpha)-\gamma^-(\alpha)}_\infty t+1\right)^d\sup_{\eta,\xi\in\Omega}\hE^E_{\eta,\xi}\rho(\eta^1_t(0),\eta^2_t(0))\,dt,
\end{align*}
where $\gamma^+(\alpha),\gamma^-(\alpha)$ are as in Proposition \ref{prop:coupling-construction} for the rates $\alpha$.

Let $Y\in\{0,1\}$ be Bernoulli with parameter $p(\alpha)$ and $Y'$ exponentially distributed with parameter $\beta$. Let $Y_1,Y_2,...$ be iid. copies of $Y\cdot Y'$ and
$N(T):=\inf\{N\geq0 : \sum_{n=1}^N Y_n>T\}$. Then
\[\Psi(T) \leq C(\alpha)\normb{f}\bE N(T) \]
\end{lemma}
\begin{proof}
By construction of the coupling, $\tau_2$ stochastically dominates $Y'$. As we have
$\tau=\tau_1\wedge\tau_2$ it follows that $\tau \succeq Y_1$.
Using this fact together with the monotonicity of $\Psi$ in \eqref{eq:PsiT},
\begin{align*}
\Psi(T) &\leq \sup_{\eta,\xi\in \Omega} \left(\hE_{\eta,\xi} \int_0^\infty f(\theta_{-X_t^1}\eta^1_t) - f(\theta_{-X_t^1}\eta^2_t) \,dt + \ind_{\tau\leq T}\Psi(T-\tau)\right) \\
&\leq \sup_{\eta,\xi\in \Omega} \hE_{\eta,\xi} \int_0^\infty f(\theta_{-X_t^1}\eta^1_t) - f(\theta_{-X_t^1}\eta^2_t) \,dt + \bE\ \ind_{Y_1\leq T}\Psi(T-Y_1).
\end{align*}
As $p(\alpha)>0$ by Lemma \ref{lemma:decoupling-estimate} we can iterate this estimate until it terminates after $N(T)$ steps. Therefore we obtain
\begin{align*}
\Psi(T) &\leq \bE N(T) \sup_{\eta,\xi\in \Omega} \hE_{\eta,\xi} \int_0^\infty f(\theta_{-X_t^1}\eta^1_t) - f(\theta_{-X_t^1}\eta^2_t) \,dt .
\end{align*}
The integral is estimated by telescoping over single site changes and Lemma \ref{lemma:supnorm-coupling-estimate} in the usual way, yielding
\[\Psi(T) \leq C(\alpha)\normb{f}\bE N(T). \]
\qed
\end{proof}

To finally come back to the original question of continuity,
\begin{align*}
\abs{\mu^{EP}_{\alpha}(f)-\mu^{EP}_{\alpha'}(f)} 
&= \frac1T \abs{ \int \int \int_0^T S_t^{EP,1}f(\eta) - S_t^{EP,2}f(\xi)\,dt\,\mu^{EP}_{\alpha}(d\eta)\,\mu^{EP}_{\alpha'}(d\xi)} \\
&\leq \frac1T \Psi(T) \leq \frac{1}{T} \bE N(T) C(\alpha) \normb{f} \\
&\konv{T\to\infty} \frac{1}{\bE YY'} C(\alpha) \normb{f}\\
&=\frac{C(\alpha)}{p(\alpha)}\sum_{z\in\bZ^d} \sup_{\eta\in\bZ^d}\abs{\alpha(\eta,z)-\alpha'(\eta,z)} \normb{f}.
\end{align*}
By sending $\alpha'$ to $\alpha$, the right hand side tends to 0 so that the ergodic measure of the environment process is indeed continuous in the rates $\alpha$. It is also interesting to note that both $p(\alpha)$ and $C(\alpha)$ are rather explicit given the original coupling of the environment. Notably when $\alpha(\eta,z)=\alpha(z)$, i.e. the rates do not depend on the environment, $p(\alpha)=1$ and $C(\alpha) = \int_0^\infty \supl_{\eta,\xi\in \Omega}\hE^E_{\eta,\xi}\rho(\eta_t^1(0),\eta_t^2(0))\,dt$.
\qed
\end{proof}

\begin{proof}[Proof of Theorem \ref{thm:EP-phi-estimate}]
The proof of this theorem is mostly identical to the proof of Theorem \ref{thm:EP-estimate}. Hence instead of copying the proof, we just state where details differ. 

A first fact is that the conditions for a) and b) imply Assumptions 1a) and b). In the adaptation of the proof for part a), in most lines it suffices to add a $\phiK$ to the integrals. However, in line \eqref{eq:phi-1}, we use 
\begin{align}
\phiK\leq \phiK[t-\tau]\phiK[\tau]	\label{eq:phiK-factor}
\end{align}
to obtain the estimate
\[ \hE_{\eta,0;\xi,0} \phiK[\tau] \int_0^{(T-\tau)\vee0} \phi(t)\abs{ S^{EP}_{t}f(\theta_{-X^1_\tau}\eta^1_\tau)-S^{EP}_{t}f(\theta_{-X^2_\tau}\eta^2_\tau) }\,dt \label{eq:phi-1} \]
instead. Thereby in lines \eqref{eq:phi-2}, \eqref{eq:phi-3} and \eqref{eq:phi-4} we have to change $\hP_{\eta,0;\xi,0}(\tau<\infty)$ to $\hE_{\eta,0;\xi,0}\phiK[\tau]\ind_{\tau<\infty}$. This change then leads to the replacement of 
\[ \infl_{\eta,\xi\in\Omega}\hP_{\eta,0;\xi,0}(\tau=\infty)\]
by the term
\[ 1-\supl_{\eta,\xi\in\Omega}\hE_{\eta,0;\xi,0}\phiK[\tau]\ind_{\tau<\infty} \]
in the lines \eqref{eq:phi-5} and \eqref{eq:phi-6} (where naturally $C_a$ becomes $C_a(K)$).
So all we have to prove that for sufficiently big $K$
\[ \supl_{\eta,\xi\in\Omega}\hE_{\eta,0;\xi,0}\phiK[\tau]\ind_{\tau<\infty} < 1. \]
In a first step, we show that 
\[ \supl_{\eta,\xi\in\Omega}\hE_{\eta,0;\xi,0}\phi(\tau)\ind_{\tau<\infty} < \infty. \]

As we already saw in the proof of Lemma \ref{lemma:decoupling-estimate}, we can view the event of decoupling as the first jump of a Poisson process with time-dependent and random rates (equation \eqref{eq:dynamic-rate-poisson}). Hence we have
\begin{align*}
\hE_{\eta,0;\xi,0}\phi(\tau)\ind_{\tau<\infty} 
&= \int_0^\infty \phi(t)\,d\hP_{\eta,0;\xi,0}(\tau>t)\\
&= \int_0^\infty \phi(t) \hE_{\eta,0;\xi,0} \suml_{z\in\bZ^d}\abs{\alpha(\theta_{-X^1_t}\eta^1_t,z)- \alpha(\theta_{-X^1_t}\eta^2_t,z)} \\
&\cdot \exp\left({-\int_0^t \suml_{z\in\bZ^d}\abs{\alpha(\theta_{-X^1_s}\eta^1_s,z)- \alpha(\theta_{-X^1_s}\eta^2_s,z)}\,ds}\right) \,dt \\
&\leq \int_0^\infty \phi(t) \hE_{\eta,0;\xi,0} \suml_{z\in\bZ^d}\abs{\alpha(\theta_{-X^1_t}\eta^1_t,z)- \alpha(\theta_{-X^1_t}\eta^2_t,z)} \,dt.
\end{align*}
By telescoping over single site discrepancies and using Lemma \ref{lemma:supnorm-coupling-estimate}, this is less than
\[ \int_0^\infty \phi(t) (\norm{\gamma^+-\gamma^-}_\infty+1)^d t^d \supl_{\eta,\xi\in \Omega} \EE_{\eta,\xi}\rho(\eta_t^1(0),\eta_t^2(0)) \,dt<\infty \]
by assumption. 
Since $\phi(t/K)$ decreases to 1 as $K\to\infty$, monotone convergence implies
\begin{align*}
 \lim_{K\to\infty} \hE_{\eta,0;\xi,0}\phiK[\tau]\ind_{\tau<\infty} = \hE_{\eta,0;\xi,0}\ind_{\tau<\infty}< 1
\end{align*}
by Lemma \ref{lemma:decoupling-estimate}. Consequently, there exists a $K_0\geq0$ such that for all $K>K_0$
\[ \hE_{\eta,0;\xi,0}\phiK[\tau]\ind_{\tau<\infty} < 1.\]
This completes the adaptation of part a).

The adaptation of the proof of part b) follows the same scheme, where we add the term $\phiK$ to all integrals. Note that this gives a version of Lemma \ref{lemma:triplenorm-coupling-estimate} as well. Then, in line \eqref{eq:phi-7} we use \eqref{eq:phiK-factor} again and then have to replace $\hP_{\eta,0;\xi,0}(\tau<\infty)$ by $\hE_{\eta,0;\xi,0}\phiK[\tau]\ind_{\tau<\infty}$ in lines \eqref{eq:phi-8} and \eqref{eq:phi-9}.
To estimate \eqref{eq:phi-9}, we use 
\begin{align*}
\hE_{\eta,0;\xi,0}\phiK[\tau]\ind_{\tau<\infty} 
&\leq \int_0^\infty \phiK[t] \hE_{\eta,0;\xi,0} \suml_{z\in\bZ^d}\abs{\alpha(\theta_{-X^1_t}\eta^1_t,z)- \alpha(\theta_{-X^1_t}\eta^2_t,z)} \,dt \\
&\leq \int_0^\infty \phiK[t] \suml_{y\in\bZ^d} w_{\alpha} \hE_{\eta,0;\xi,0} \rho(\eta^1_t(y+X_t^1),\eta^2_t(y+X_t^1)) \,dt 
\end{align*}
with $w_\alpha$ as in the original proof. Therefore
\begin{align*}
&\sumsup \frac{\hE_{\eta,\xi}\phiK[\tau]\inf_{\tau<\infty}}{\rho(\eta(x),\xi(x))}  \\
&\quad\leq	\sumsup\int_0^\infty \phiK \suml_{y\in\bZ^d} w_\alpha(y) \frac{\hE_{\eta,\xi}\rho(\eta_t^1(y+X_t^1), \eta_t^2(y+X_t^1))}{\rho(\eta(x),\xi(x))} \,dt, \\
\end{align*}
which is finite by Lemma \ref{lemma:triplenorm-coupling-estimate}.\qed
\end{proof}

\section{Law of Large Numbers and Einstein Relation}\label{section:LLN}
\subsection{Law of Large Numbers for the position of the walker}
\begin{theorem}\label{thm:LLN}
Under Assumption 1a,
\[ \lim_{T\to\infty}\frac{X_T}{T} = \int \suml_{z\in\bZ^d} z\alpha(\eta,z)\,\mu^{EP}(d\eta) \]
in $L_1(\bP_{\nu,0})$ and $\bP_{\nu,0}$-a.s. for any probability measure $\nu$ on $\Omega$.
\end{theorem}
\begin{proof}
First we assume that the only shifts the environment process $\eta_t^{EP}=\theta_{-X_t}\eta_t$ performs are induced by jumps of the walker $X_t$. Assumption 1a does not prohibit the environment to perform shifts itself, but the argument comes more natural when is does not.

Let $F_z : D([0,1],\Omega)\to\bR$, $z\in\bZ^d$ count the number of shifts of size $z$ a piece of trajectory performs in the interval $[0,1]$, i.e.
\[ F_z(\eta^{EP}_{[t,t+1]}) = \suml_{s\in]0,1]} \ind_{\theta_{-z}\eta^{EP}_{t+s}=\eta^{EP}_{t+s-}}. \]
With $F := \sum_{z\in\bZ^d}zF_z$, which is well-defined and in $L^1(\mu^{EP})$ because of the rate condition $\norm{\alpha}_1<\infty$, we then have for any integer $T>0$
\[ X_T - X_0 = \sum_{n=1}^T F(\eta^{EP}_{[n-1,n]}). \]
The ergodic theorem then implies
\[ \lim_{T\to\infty} \frac{X_T - X_0}{T} = \lim_{T\to\infty}\frac{1}{T}\sum_{n=1}^T F(\eta^{EP}_{[n-1,n]}) = \mu^{EP}(F). \]
The same is true for non-integer $T$, as we simply use the fact that $X_T-X_{\lfloor T \rfloor}$ has bounded expectation.
Since 
\begin{align*}
 \mu^{EP}(F) &= \int \int_0^1 \bE_{\eta,0} \suml_{z\in\bZ^d} z\alpha(\theta_{-X_t}\eta_t,z)\,dt\,\mu^{EP}(d\eta)\\
 & = \int \suml_{z\in\bZ^d} z\alpha(\eta,z)\,\mu^{EP}(d\eta),
\end{align*}
the claim is proven for $\mu^{EP}$ if the environment performs no shifts.

In the case that the environment does perform shifts, we introduce an additional auxiliary counting processes  $N_t^z$ on $\bN$, $z\in\bZ^d$, where $N_t^z$ increases by one whenever $X_t-X_{t-}=z$ and $N_t^z$ jumps to 0 at rate one to make it stationary. Then $F_z$ counts the number of increments of $N_t^z$ instead of the number of shifts. The rest of the proof is the same.

To extend the result to an arbitrary probability measure $\nu$ we use a coupling argument. We look at $\abs{X^1_T-X^2_T}$ under $\hP_{\nu,0;\mu^{EP},0}$, where we slightly modifiy the coupling from Proposition \ref{prop:coupling-construction} so that when the two walkers decouple at a time $\tau$ we restart the coupling of the environment. This restart is done on $\theta_{-X^1_\tau}\eta^{1}_{\tau},\theta_{-X^2_\tau}\eta^{2}_{\tau}$ so that after the decoupling it acts as if both walkers are back at the origin. Then, by Lemma \ref{lemma:decoupling-estimate}, there is at least probability $p>0$ to never decouple initially or after a decoupling. So there is at most a geometric number $N$ of decoupling events at $\tau_1,..\tau_N$, and only at decoupling events the two walkers perform different jumps(one jumps, the other does not). Hence 
\[ \hE_{\nu,0;\mu^{EP},0} \abs{X^1_T-X^2_T} \leq \hE_{\nu,0;\mu^{EP},0} \suml_{n=1}^N (\abs{X^1_{\tau_n}-X^1_{\tau_n-}}+\abs{X^2_{\tau_n}-X^2_{\tau_n-}}), \]
which converges to 0 in $L^1$ and almost surely when diveded by $T$ and sending $T$ to infinity.
\qed
\end{proof}

\subsection{Weak interaction with the environment and an Einstein Relation}
It can be interesting to consider random walks which are only weakly affected by an external influence. In the context we study here, that corresponds to rates $\alpha$ which are only weakly dependent on the state $\eta$ of the environment. 

To study this weak interaction, let $\alpha_\epsilon : \Omega \times \bZ^d \to[0,\infty[\ , 0\leq\epsilon$ be rates for the walker, satisfying the conditions from Section \ref{subsection:rates}. Furthermore we assume that $\epsilon \mapsto \alpha_\epsilon(\eta,z)$ is differentiable in 0 and $\alpha_0(\eta,z)=\alpha_0(z)$, that is for $\epsilon=0$ the walker is independent of the environment. Finally we assume that $\normb{\alpha_\epsilon-\alpha_0}_1 \leq C\epsilon$ for some constant $C>0$.

Denote with $v_\epsilon$ the asymptotic speed of the walker with rates $\alpha_\epsilon$.
\begin{theorem}\label{thm:second-order}
The change of the ergodic measure of the environment process from $\mu^E$, the ergodic measure of the environment, to $\mu^{EP}_\epsilon$, the ergodic measure corresponding to $\alpha_\epsilon$, is a second order influence on the speed $v_\epsilon$ of the walker:
\[ \abs{v_\epsilon - \int \suml_{z\in\bZ^d} z \alpha_\epsilon(\eta,z) \,\mu^E(d\eta)} \leq C' \epsilon^2, \]
where $C'$ is a constant independent of $\alpha_\epsilon$, $\epsilon\geq0$.
\end{theorem}
\begin{proof}
By Theorem \ref{thm:LLN},
\[ v_\epsilon = \int \suml_{z\in\bZ^d} z \alpha_\epsilon(\eta,z)\,\mu^{EP}_\epsilon(d\eta) . \]
Note that $\mu^{EP}_0 = \mu^E$. By Theorem \ref{thm:mu-continuous},
\begin{align*}
&\abs{\int \suml_{z\in\bZ^d} z \alpha_\epsilon(\eta,z)\,\mu^{EP}_\epsilon(d\eta) - \int \suml_{z\in\bZ^d} z \alpha_\epsilon(\eta,z)\,\mu^{EP}_0(d\eta)} \\
&\quad\leq C(\alpha_0)\norm{\alpha_\epsilon-\alpha_0}_0 \normb{\suml_{z\in\bZ^d}z\alpha_\epsilon(\cdot,z)}
\end{align*}
and $C(\alpha_0)$ is in fact independent of $\alpha_0$ as those rates are independent of the environment. The estimates $\norm{\alpha_\epsilon-\alpha_0}_0 \leq \normb{\alpha_\epsilon-\alpha_0}_1 \leq C\epsilon$ and
\[ \normb{\suml_{z\in\bZ^d}z\alpha_\epsilon(\cdot,z)} \leq \normb{\alpha_\epsilon}_1 = \normb{\alpha_\epsilon-\alpha_0}_1 \leq C\epsilon\]
complete the proof.
\qed
\end{proof}

An \emph{Einstein Relation} is said to hold if the change of speed of the walker because of the external influence is equal to the diffusion constant of the walker in the limit of the influence going to 0, or more precisely
\begin{align}\label{eq:ER}
\lim_{\epsilon\to 0} \frac{v_\epsilon-v_0}{\epsilon} = \sigma_0^2
\end{align}
with $\sigma_0^2 = \sum_{z\in\bZ}z^2\alpha_0(z)$. Note that the typical formulation is in the case $v_0 = 0$.
\begin{corollary}
The Einstein Relation \eqref{eq:ER} holds iff
\[ \suml_{z\in\bZ} z \int \alpha'_0(\eta,z) \,\mu^E(d\eta) = \sum_{z\in\bZ}z^2\alpha_0(z), \]
where $\alpha'_0(\eta,z)$ is the derivative in $\epsilon$. 

Especially, the Einstein Relation holds if
\[ \int \alpha'_0(\eta,z) \,\mu^E(d\eta) = z\alpha_0(z) \quad\forall\,z\in\bZ. \]
\end{corollary}
\begin{proof}
By Theorem \ref{thm:second-order}, 
\[ v_\epsilon = \int \suml_{z\in\bZ^d} z \alpha_\epsilon(\eta,z) \,\mu^E(d\eta) + O(\epsilon^2) \]
and hence 
\[ \lim_{\epsilon\to 0} \frac{v_\epsilon-v_0}{\epsilon} = \suml_{z\in\bZ} z \int \alpha'_0(\eta,z) \,\mu^E(d\eta). \hqed\]
\end{proof}

\section{Central Limit Theorem}\label{section:CLT}

Before starting with the functional central limit theorem for the random walk $(X_t)$ in its random dynamic environment let us quickly discuss the CLT for additive functionals of the environment process. 

Let us assume that $f:\Omega\to\bR$ with $\normb{f}<\infty$ and $\mu^{EP}(f)=0$. Then
\[ \left(L^{EP}\right)^{-1}f(\eta) = -\int_0^\infty S_t^{EP}f(\eta)\,dt, \]
where Theorem \ref{thm:EP-estimate} guarantees that the right hand side is well defined. Hence standard arguments provide a functional CLT for additive functionals $\int_0^T f(\eta^{EP}_t)\,dt$ in this context. 

More work is necessary for a CLT of the position of the walker. The standard approach is to use the fact that the difference between the position of the walker $X_T$ and the cumulate rates $\int_0^T \sum_{z\in\bZ^d}z\alpha(\theta_{-X_t}\eta_t,z)\,dt$ is a martingale. One can prove a CLT for this martingale, and together with the CLT for the drift-adjusted integral obtain a CLT for the position. However, one has to take care that that the variances do not annihilate, and generally there is no explicit formulation of the variance. As the martingales used in the appendix provide an alternative method for obtaining a CLT which does include an explicit formulation of the variance, we explore that approach.
\subsection{A second assumption on the environment}\label{section:assumption2}
There is one additional assumption on the environment required to easily use the martingales(and the moment estimates they provide). 
\begin{assumption}[2]
There exists a constant $R^E\geq0$ such that for any $f : \Omega \to \bR$ with $\normb{f}<\infty$,
\begin{align}\label{eq:assumption-2}
 \oLE(f-f(\eta))^2(\eta) := \limsup_{\epsilon\to0}\frac{1}{\epsilon}S^E_\epsilon (f-f(\eta))^2(\eta) \leq R^{E}\normb{f}^2 .
\end{align}
\end{assumption}
This assumption might look artificial but the following lemma shows that Assumption 2 is in fact well behaved.
\begin{lemma}\label{lemma:A2-fg}
Assume Assumption 2. Then, for any $f,g: \Omega\to\bR$ with $\normb{f},\normb{g}<\infty$, the corresponding more general estimate holds:
\[ \abs{\oL^E[(f-f(\eta))(g-g(\eta))](\eta)}\leq R^E\normb{f}\normb{g}. \]
\end{lemma}
\begin{proof}
It is a direct consequence of Lemma \ref{lemma:chauchy-schwarz} and Assumption 2.
\end{proof}
Now we extend Assumption 2 from the environment to the environment process.
\begin{lemma}\label{lemma:EP-A2}
Under Assumption 2, for any $f : \Omega \to \bR$ with $\normb{f}<\infty$,
\[ \oLEP(f-f(\eta))^2(\eta) := \limsup_{\epsilon\to0}\frac{1}{\epsilon}S^{EP}_\epsilon (f-f(\eta))^2(\eta) \leq R^{EP}\normb{f}^2 \]
with $R^{EP}=R^E+\normb{\alpha}+2\norm{\alpha}_0$.
\end{lemma}
\begin{proof}
As the generator of the environment process  is $L^{EP} = L^E + L^{S}$, with 
\[ L^{S}f(\eta) = \suml_{z\in\bZ^d}\alpha(\eta,z)\left[f(\theta_{-z}\eta)-f(\eta)\right] \]
being the shifts induced by the random walk jumps, all we need to proof is that $L^{S}$ is a bounded operator in terms of the $\normb{\cdot}$-norm:
\begin{align*}
&\sumsup \left(L^{S}f(\eta)-L^{S}f(\xi)\right) \\
&\quad= \sumsup \suml_{z\in\bZ^d}\left(\alpha(\eta,z)\left[f(\theta_{-z}\eta)-f(\eta)\right] - \alpha(\xi,z)\left[f(\theta_{-z}\xi)-f(\xi)\right] \right) \\
&\quad\leq \sumsup \suml_{z\in\bZ^d}\left(\alpha(\eta,z)-\alpha(\xi,z)\right)\left[f(\theta_{-z}\eta)-f(\eta)\right] \\
&\qquad + \sumsup \suml_{z\in\bZ^d}\alpha(\xi,z)\left[f(\theta_{-z}\eta)-f(\eta) - f(\theta_{-z}\xi)+f(\xi)\right] . 
\end{align*}
The first term is easily estimated by $\normb{\alpha}\norm{f}_{osc}\leq\normb{\alpha}\normb{f}$. For the second term, it is smaller than
\begin{align*}
	\suml_{z\in\bZ^d}\sup_{\xi\in \Omega}\alpha(\xi,z) 2\normb{f} =: 2\norm{\alpha}_{0}\normb{f}.
\end{align*}
So in total we get
\[ \norm{L^{S}}_{\normb{\cdot}\to\normb{\cdot}} \leq \normb{\alpha}+2\norm{\alpha}_0 \]
\qed
\end{proof}
\subsection{CLT for the path of the random walk}
The CLT for the path of the random walk is proven in two steps. First we show that a martingale which corresponds to fluctuations in a specific direction converges to Brownian motion. Secondly, we we use the convergence of those martingales to obtain the full functional central limit theorem.
\begin{proposition}\label{prop:martingale-CLT}
Assume Assumption 1b and 2 as well as $\norm{\alpha}_2<\infty$.
Fix $v\in\bR^d$ with $\norm{v}=1$ and write $f:\Omega\times\bZ^d \to \bR, f(\eta,x)=\left<x,v\right>$.
Define for each $T>0$ the martingales
\[ M_T(t) := T^{-\frac12}\left(\erwc{f(\eta_T,X_T)}{\fF_{tT}}-\erwc{f(\eta_T,X_T)}{\fF_{0}}\right),\quad 0\leq t\leq T. \]
Then $(M_T(t))_{0 \leq t\leq 1}$ converges to Brownian motion with variance
\[ \int L\left(\Phi_{0,\infty}-\Phi_{0,\infty}(\eta,0) + f-f(\eta,0) \right)^2(\eta,0) \,\mu^{EP}(d\eta), \]
where
\[ \Phi_{0,\infty}-\Phi_{0,\infty}(\eta,0) = \suml_{z\in\bZ^d}<z,v> \int_0^\infty S_t^{EP}[\alpha(\cdot,z)]-S_t^{EP}[\alpha(\cdot,z)](\eta) \,dt. \]
\end{proposition}
\begin{proof}
By Proposition \ref{prop:point-quadratic-variation}, the predictable quadratic variation is given by
\[ \left<M_T\right>_t = \frac1T \int_0^{tT} L\left(S_{T-s}f-S_{T-s}f(\eta_s,X_s)\right)^2(\eta_s,X_s)\,ds. \]
We want to take the limit $T\to\infty$, and to that end we first rewrite the inner term via
\begin{align*}
 S_t f(\eta,x) &= \int_0^t S_s^{EP} \suml_{z\in\bZ^d} [\alpha(\cdot,z)\left<z,v\right>](\theta_{-x}\eta) \,ds + \left<x,v\right>\\
 &= \Phi_{0,t}(\eta,x) + f(\eta,x)
 \end{align*}
with
\[ \Phi_{t,t'}(\eta,x) := \suml_{z\in\bZ^d}\left<z,v\right>\int_t^{t'} S_s^{EP} [\alpha(\cdot,z)](\theta_{-x}\eta) \,ds \]
for shorter notation. To replace $\Phi_{0,t}$ by $\Phi_{0,\infty}$, we observe that
\begin{align*}
& | L\left[\Phi_{0,t}-\Phi_{0,t}(\eta,x) + f-f(\eta,x)\right]^2(\eta,x) \\
&\qquad- L\left[\Phi_{0,\infty}-\Phi_{0,\infty}(\eta,x) + f-f(\eta,x)\right]^2(\eta,x) |\\
&= \left| L\left[(\Phi_{t,\infty}-\Phi_{t,\infty}(\eta,x))(\Phi_{0,t}- \Phi_{0,t}(\eta,x) + \Phi_{0,\infty} - \Phi_{0,\infty}(\eta,x) + 2f-2f(\eta,x)\right](\eta,x) \right| \\
&\leq \left(\oL\left[\Phi_{t,\infty}-\Phi_{t,\infty}(\eta,x)\right]^2(\eta,x)\right)^{\frac12} \\
&\qquad \cdot \left(\oL\left[\Phi_{0,t}- \Phi_{0,t}(\eta,x) + \Phi_{0,\infty} - \Phi_{0,\infty}(\eta,x) + 2f-2f(\eta,x)\right]^2(\eta,x)\right)^{\frac12},
\end{align*}
where the last line is a consequence of Lemma \ref{lemma:chauchy-schwarz}.
As $\Phi$ is effectively a function of the environment process $\theta_{-x}\eta$, we can use Lemma \ref{lemma:EP-A2} for the estimate
\[ \oL\left[\Phi_{t,\infty}-\Phi_{t,\infty}(\eta,x)\right]^2(\eta,x)\leq R^{EP}\normb{\Phi_{t,\infty}}^2. \]
The second factor is estimated similarly after first using $(a+b)^2\leq 2a^2+2b^2$ by
\begin{align*}
&2 R^{EP}\normb{\Phi_{0,t}+\Phi_{0,\infty}}^2 + 8 \oL(f-f(\eta,x))^2(\eta,x) \\
&\leq R^{EP}(\normb{\Phi_{0,t}}+\normb{\Phi_{0,\infty}})^2 + 8 \suml_{z\in\bZ^d}\alpha(\theta_{-x}\eta,z)\left<z,v\right>^2.
\end{align*}
By the second moment assumption $\norm{\alpha}_2<\infty$ the right summand is bounded, and using Theorem \ref{thm:EP-estimate}, $\normb{\Phi_{0,t}}\leq C_b \normb{\alpha}_1$, and $\normb{\Phi_{t,\infty}}\konv{t\to\infty}0$. Coming back to the quadratic variation of the martingale $M_T$, we can now conclude that
\begin{align*}
\lim_{T\to\infty} \left<M_T\right>_t &= \lim_{T\to\infty} \frac1T \int_0^{tT} L\left(S_{T-s}f-S_{T-s}f(\eta_s,X_s)\right)^2(\eta_s,X_s)\,ds \\
&= \lim_{T\to\infty} \frac1T \int_0^{tT} L\left(\Phi_{0,\infty}-\Phi_{0,\infty}(\eta_s,X_s) + f-f(\eta_s,X_s) \right)^2(\eta_s,X_s)\,ds \\
&= \lim_{T\to\infty} \frac1T \int_0^{tT} L\left(\Phi_{0,\infty}-\Phi_{0,\infty}(\theta_{-X_s}\eta_s,0) + f-f(\theta_{-X_s}\eta_s,0) \right)^2(\theta_{-X_s}\eta_s,0)\,ds \\
&= t\int L\left(\Phi_{0,\infty}-\Phi_{0,\infty}(\eta,0) + f-f(\eta,0) \right)^2(\eta,0) \,\mu^{EP}(d\eta) .
\end{align*}
Now that we have the converge of the predictable quadratic variation, all that remains to obtain the CLT for the martingale $M_T$ is to show that its jumps vanish in the limit:
\[ \lim_{T\to\infty}\bE\sup_{0\leq t\leq 1} \left(M_T(t)-M_T(t-)\right)^2 =0. \]
To obtain this, we use that
\begin{align*}
 &M_T(t)-M_T(t-) \\
 &\quad= T^{-\frac12}\left(S_{T-tT}f(\eta_{tT},X_{tT})-S_{T-tT}f(\eta_{tT-},X_{tT-})\right) \\
 &\quad= T^{-\frac12}\left(\Phi_{0,T-tT}(\eta_{tT},X_{tT})-\Phi_{0,T-tT}(\eta_{tT-},X_{tT-}) + f(\eta_{tT},X_{tT})-f(\eta_{tT-},X_{tT-})\right) \\
 &\quad= T^{-\frac12}\left(\Phi_{0,T-tT}(\eta_{tT},X_{tT})-\Phi_{0,T-tT}(\eta_{tT-},X_{tT-}) + \left<X_{tT}-X_{tT-},v\right>)\right). \\
\end{align*}
As we already now that the difference of the two $\Phi$'s is uniformly bounded, it is only necessary to 
look at the jumps of $X_t$:
\begin{align*}
&\frac1T\bE\supl_{0\leq t\leq T}\norm{X_t-X_{t-}}^2 \\
&\quad\leq \frac{\lambda^2}{T}\bP(\text{all jumps }\leq \lambda) + \frac1T \bE\supl_{0\leq t\leq T}\norm{X_t-X_{t-}}^2\ind_{\{\text{some jump }> \lambda\}} \\
&\quad\leq \frac{\lambda^2}{T} + \frac1T \suml_{\norm{z}>\lambda} \norm{z}^2 \bP(\exists\ 0\leq t \leq T : X_t -X_{t-}=z) \\
&\quad\leq \frac{\lambda^2}{T} + \frac1T \suml_{\norm{z}>\lambda} \norm{z}^2 T \sup_{\eta \in \Omega}\alpha(\eta,z) \\
&\quad= \frac{\lambda^2}{T} + \suml_{\norm{z}>\lambda} \sup_{\eta \in \Omega}\alpha(\eta,z)\norm{z}^2. \\
\end{align*}
By the second moment condition $\norm{\alpha}_2<\infty$, the sum on the right converges to 0 as $\lambda\to\infty$. So when we choose for example $\lambda=T^{\frac14}$, the right hand side converges to 0 as $T\to\infty$, which proves the functional CLT for $(M_T(t))_{0\leq t\leq 1}$.
\qed
\end{proof}

\begin{theorem}\label{thm:CLT}
Assume Assumption 1b and 2 and $\norm{\alpha}_2<\infty$. Let
\[ v=\int\suml_{z\in\bZ^d}\alpha(\eta,z)z\,\mu^{EP}(d\eta) \]
be the asymptotic speed of the random walk $(X_t)_{t\geq0}$. Then
\[ \left((T^{-\frac12} (X_{tT}-vtT)\right)_{0\leq t\leq 1} \]
converges in probability to a Brownian motion with covariance matrix
\begin{align}\label{eq:CLT-var}
 \Sigma^2 = \int L\left[\left( A_\eta + id_x\right) \left(A_\eta + id_x\right)^T\right](\eta,0)\,\mu^{EP}(d\eta), 
\end{align}
where $A_\eta : \Omega\times\bZ^d \to \bR^d$ with
\[ A_\eta(\xi,x) := \suml_{z\in\bZ^d}z\int_0^\infty S_t^{EP}[\alpha(\cdot,z)](\theta_{-x}\xi)-S_t^{EP}[\alpha(\cdot,z)](\eta)\,dt \]
and $id_x(\xi,x)=x$.
\end{theorem}
\begin{proof}
Define the $d$-dimensional martingale 
\[ M_T(t) := T^{-\frac12}\left( \erwc{X_T}{\fF_{tT}}-\erwc{X_T}{\fF_0} \right). \]
By Proposition \ref{prop:martingale-CLT} the projection onto any unit vector $u\in\bR^d$ converges to a
Brownian motion with variance
\begin{align*}
&\int L\left( <A_\eta,u> + <id_x,u> \right)^2(\eta,0)\,\mu^{EP}(d\eta) \\
&\quad= u^T \int L\left[\left( A_\eta + id_x\right) \left(A_\eta + id_x\right)^T\right](\eta,0)\,\mu^{EP}(d\eta) \,u.
\end{align*}
As the projected martingales are adapted to $(\fF_t)$, the $\sigma$-algebra of $M_T$, that implies that $M_T$ converges to a Brownian motion with the given covariance matrix. Since $M_T(1)=T^{-\frac12}\left( X_T - \bE_{\eta_0,0}X_T \right)$ we can already conclude a central limit theorem. To obtain the functional central limit theorem, we simply use the fact that $M_T(t)$ is close to $T^{-\frac12}\left(X_{tT}-vtT\right)$:
\begin{align*}
T^{\frac12}M_T(t) 
&= \erwc{X_T}{\eta_{tT},X_{tT}} - \erwc{X_T}{\eta_{0},X_{0}} \\
&= \int_0^{T-tT} \suml_{z\in\bZ^d}z S_s^{EP}[\alpha(\cdot,z)](\theta_{-X_{tT}}\eta_{tT})+X_{tT} \\
&\qquad - \int_0^T\suml_{z\in\bZ^d}z S_s^{EP}[\alpha(\cdot,z)](\eta_{0}).
\end{align*}
When we split the last integral at $tT$, we notice that the part up to $tT$ is close to $vtT$:
\begin{align*}
&\abs{\int_0^{tT}\suml_{z\in\bZ^d}z S_s^{EP}[\alpha(\cdot,z)](\eta_{0}) - vtT} \\
&\quad= \abs{\int_0^{tT}\suml_{z\in\bZ^d}z S_s^{EP}[\alpha(\cdot,z)](\eta_{0})-\int \int_0^{tT}\suml_{z\in\bZ^d}z S_s^{EP}[\alpha(\cdot,z)](\xi)\,\mu^{EP}(d\xi)} \\
&\quad \leq C_a \normb{\alpha}_1
\end{align*}
by Theorem \ref{thm:EP-estimate}. Similarly, the part after $tT$ almost annihilates with the integral from $0$ to $T-tT$:
\begin{align*}
&\abs{\int_0^{T-tT} \suml_{z\in\bZ^d}z S_s^{EP}[\alpha(\cdot,z)](\theta_{-X_{tT}}\eta_{tT}) - \int_{tT}^T\suml_{z\in\bZ^d}z S_s^{EP}[\alpha(\cdot,z)](\eta_{0})} \\
&\quad= \abs{\int_0^{T-tT} \suml_{z\in\bZ^d}z S_s^{EP}[\alpha(\cdot,z)](\theta_{-X_{tT}}\eta_{tT}) - \bE_{\eta_0,0} \int_{0}^{T-tT}\suml_{z\in\bZ^d}z S_s^{EP}[\alpha(\cdot,z)](\theta_{-X_{tT}}\eta_{tT})} \\
&\quad \leq C_a\normb{\alpha}_1.
\end{align*}
So 
\[ \abs{M_T(t) - \frac{X_tT-vtT}{T^{\frac12}}} \leq \frac{2C_a\normb{\alpha_1}}{T^{\frac12}}, \]
and the convergence of the drift-adjusted random walk follows to Brownian motion follows from the convergence of $M_T$.\qed
\end{proof}
\subsection{Some remarks on the variance}
A first comment is that the variance \eqref{eq:CLT-var} is indeed non-degenerate. We remember that the generator $L$ acts as a positive operator, as $A_\eta(\eta,0) = 0$ and $id_x(\eta,0)=0$. By splitting $L = L^E + L^{RW}$, we furthermore see that either $L^E \left[A_\eta A_\eta^T\right](\eta)$ is $\mu^{EP}$-a.s. 0, which implies $A_\eta =0$, or the variance is already positive by contributions from the environment alone. If $A_\eta=0$, then the second half becomes 
\[ L^{RW}\left[id_x id_x^T\right](\eta,0) = \suml_{z\in \bZ^d} z z^T \alpha(\eta,z), \]
which is 0 if and only of the rates $\alpha$ are 0 $\mu^{EP}$-a.s., in which case the walker is degenerate by construction.

However we can also use the explicit formulation of the variance to study the behaviour as the speed of the environment is increased. For a positive $\lambda$ denote by
\[ L^\lambda = \lambda L^E + L^{RW} \]
the generator where the environment runs at speed $\lambda$. By noticing that the new system corresponds to one where the rates are scaled by $1/\lambda$ plus a rescaling of time by $\lambda$ (or by following the proof of Theorem \ref{thm:EP-estimate}) one can see that $A^\lambda_\eta \in O(\lambda^{-1})$. Hence
\[ \lim_{\lambda\to\infty} \lambda L^E \left[A^{\lambda}_\eta (A^{\lambda}_\eta)^T \right](\eta) = 0 \]
and 
\[ \lim_{\lambda\to\infty} L^{RW} \left[(A^{\lambda}_\eta+id_x) (A^{\lambda}_\eta+id_x)^T \right](\eta) = L^{RW} \left[id_x id_x^T \right](\eta) = \suml_{z\in\bZ^d} z z^T \alpha(\eta,z). \]
As the ergodic measure of the environment process converges to $\mu^E$, the ergodic measure of the environment, we obtain that the variance $\Sigma_\lambda^2$ converges to 
\[ \int \suml_{z\in\bZ^d} z z^T \alpha(\eta,z)\,\mu^E(d\eta), \]
the variance when averaging of the environment.

\section{Concentration estimates}\label{section:RWRE-concentration}
In this section we will obtain detailed results about the deviation from the mean of additive functionals of the environment process and of the position of the walker itself. It is comparatively easy to obtain those from general methods.

Let us do a quick overview of the general ingredients we use here to obtain concentration of additive functionals of Markov processes. Those are general and not specific to this setting. One part is the existence of a suitable estimate of the form
\begin{align}\label{eq:concentration-condition-1}
[L(g-g(y))^2](y) \leq R \normb{g}^2. 	
\end{align}
 The second part is an estimate of
\begin{align}\label{eq:concentration-condition-2}
\normb{\int_0^T S_t f\,dt} \leq C_T \normb{f}.	
\end{align}
In general, an estimate of $\norm{\int_0^T S_tf\,dt}_{osc}$ is also convenient, but in our case that is implied, as $\norm{g}_{osc}\leq \normb{g}$ (note that $\normb{\cdot}$ could be any kind of norm-like object in a different context). If estimates \eqref{eq:concentration-condition-1} and \eqref{eq:concentration-condition-2} are satisfied, then concentration estimates follow for additive functionals $\int_0^T f(X_t)\,dt$. 
The reader is referred to \cite{} for a more in-depth look at those methods. Section \ref{section:concentration} contains similar methods to obtain concentration for $f(X_T)$ instead of additive functionals.

In our context, Assumption 1b guarantees \eqref{eq:concentration-condition-2} for the environment, and Theorem \ref{thm:EP-estimate},b) lifts that to the environment process. Similarly, Assumption 2 provides \eqref{eq:concentration-condition-1} for the environment and Lemma \ref{lemma:EP-A2} for the environment process.

\subsection{Concentration of additive functionals of the environment process}
\begin{theorem}\label{thm:concentration-EP-exp}
Assume Assumptions 1b and 2. Then there exists a constant $D>0$ so that for any $f:\Omega \to \bR$ with $\normb{f}<\infty$ the deviation estimate
\[ \bP^{EP}_{\nu_1} \left(\int_0^T f(\eta^{EP}_t)\,dt > C_a\normb{f} (r+1) + \int_0^T \nu_2 \left(S^{EP}_t f\right)\,dt\right) 
\leq e^{-\frac{ r^2}{TD}}
\]
holds for any $r\geq0$ and any probability measures $\nu_1,\nu_2$ on $\Omega$. In the case that $\nu_1,\nu_2$ are equal and point measures, the same holds with $(r+1)$ replaced by $r$.
\end{theorem}
The idea of the proof is to use Lemma \ref{lemma:exponential-concentration}. Before starting with the proof, we state a fact which is highly useful to prove the conditions of the general concentration estimates.
\begin{lemma}\label{lemma:oL-trick}\mbox{}
\begin{enumerate}
\item Let $f: \Omega\to\bR$ with $\norm{f}_{osc}<\infty$. Then
\[ \oL[(f-f(\eta))^k](\eta) \leq \norm{f}_{osc}^{k-2}\oL(f-f(\eta))^2(\eta); \]
\item Fix $\eta\in \Omega$, and let $g_\eta, f_\eta:E\to\bR$ satisfy $f_\eta(\eta),g_\eta(\eta)=0$ and $f_\eta\leq g_\eta$. Then
\[ \oL f_\eta \leq \oL g_\eta . \]
\end{enumerate}
\end{lemma}
\begin{proof}
a) Is a direct consequence of b) with $f_\eta = (f-f(\eta))^k$ and $g_\eta = \norm{f}_{osc}^{k-2}(f-f(\eta))^2$.

b) Follows by definition of $\oL$: $\oL f_\eta = \limsup_{\epsilon\to0}\frac1\epsilon S_\epsilon f_\eta \leq \limsup_{\epsilon\to0} \frac1\epsilon S_\epsilon g_\eta = \oL g_\eta$.
\qed
\end{proof}
\begin{remark}
This lemma is completely general and not restricted to this context.
\end{remark}
\begin{proof}[Proof of Theorem \ref{thm:concentration-EP-exp}]
We want to apply Lemma \ref{lemma:exponential-concentration}. So we first prove that 
\begin{align}\label{eq:proof-1}
\oLEP \left(\int_0^{T'} S^{EP}_t f - S^{EP}_tf(\eta)\,dt\right)^k(\eta) < D'\left(C_a\normb{f}\right)^k	
\end{align}
for any $T'>0, \eta \in \Omega$ and $k\in \bN, k\geq2$.
Recall that by Theorem \ref{thm:EP-estimate} part a),
\[ \norm{\int_0^{T'} S_t^{EP}f}_{osc} \leq C_a \normb{f}. \]
Therewith, using Lemma \ref{lemma:oL-trick} a),
\begin{align*}
 &\oLEP \left(\int_0^{T'} S^{EP}_t f - S^{EP}_tf(\eta)\,dt\right)^k(\eta) \\
 &\quad\leq \left(C_a\normb{f}\right)^{k-2} \oLEP \left(\int_0^{T'} S^{EP}_t f - S^{EP}_tf(\eta)\,dt\right)^2(\eta).
\end{align*}
By Lemma \ref{lemma:EP-A2},
\begin{align*}
&\oLEP \left(\int_0^{T'} S^{EP}_t f - S^{EP}_tf(\eta)\,dt\right)^2(\eta) \\
&\quad\quad \leq R^{EP}\normb{\int_0^{T'} S^{EP}_t f - S^{EP}_tf(\eta)\,dt }^2 .
\end{align*}
Using Theorem \ref{thm:EP-estimate} part b) for the estimate 
\[ \normb{\int_0^{T'} S^{EP}_t f - S^{EP}_tf(\eta)\,dt}\leq C_b\normb{f}, \]
we choose $D'= \frac{R^{EP}C_b^2}{C_a^2}$ to obtain \eqref{eq:proof-1}. As a direct consequence, Lemma \ref{lemma:exponential-concentration} implies
\[ \bP^{EP}_{\nu_1} \left(\int_0^T f(\eta^{EP}_t)\,dt > C_a\normb{f} (r+1) + \int_0^T \nu_2 \left(S^{EP}_t f\right)\,dt\right) 
\leq e^{-\frac{ \frac12r^2}{TD'+\frac13r}}.
\]
Since $\normb{f}\geq \norm{f}_{osc}$, the left-hand-side is in fact 0 for $r\geq \frac{T}{C_a}$. So we can further estimate that probability by replacing $\frac13 r$ by $\frac{T}{3C_a}$. Choosing 
$D= 2D'+\frac{2}{3C_a}$ completes the proof.
 \qed 
\end{proof}
\subsection{Concentration estimates for the position of the walker}
Note how Theorem \ref{thm:concentration-EP-exp} already gives us a strong concentration property for the rates of the walker $X_t$. 
However, it is a bit more tricky to obtain good concentration estimates for the position of the walker itself.

For concentration results for the position of the walker, we need moment conditions on the transition rates of the walker which are comparable to the strength of the concentration estimate. For example, if the jumps are of bounded size or have some exponential moment, then we obtain a concentration estimate which is Gaussian for small and exponential for large deviations. Note that it is impossible to obtain pure Gaussian concentration, as the example of jumps with rate 1 of size 1 to the right, independent of the environment, i.e. a Poisson process, shows.
\begin{theorem}\label{thm:concentration-walker-exp}
Assume Assumptions 1b and 2. Also assume that $\normb{\alpha}_1<\infty$ and that there exists some $M<\infty$ so that the bound $\norm{\alpha}_k \leq M^k$ is satisfied for any $k\in\bN$. Then there exist constants $c_1,c_2>0$ so that for any probability distributions $\nu_1,\nu_2$ on $\Omega$ and any $r>0$
\[ \bP_{\nu_1,0}\left( \norm{X_T- \bE_{\nu_2,0}X_T}> c_1 r+2dC_a\normb{\alpha}_1\right) \leq 2de^{-\frac{\frac12 r^2}{Tc_2 + \frac13r}}. \]
\end{theorem}
If only some moments for the jumps of the walker exist, then we still get a concentration estimate, although it is correspondingly weaker.
\begin{theorem}\label{thm:concentration-walker-moment}
Assume Assumptions 1b and 2. Fix $p>1$. Assume that $\normb{\alpha}_1<\infty$ and $\norm{\alpha}_p < \infty$. Then there exist a constant $c>0$ depending on the process, but not on $p$, and a constant $c_p$ depending only on $p$ so that for any probability distributions $\nu_1,\nu_2$ on $\Omega$ and any $r>0$
\[ \bP_{\nu_1,0}\left( \norm{X_T- \bE_{\nu_2,0}X_T}> r\right) \leq c_p \frac{c^p(T^{\frac p2}+1)}{r^p}. \]
\end{theorem}
The remainder of this section will deal with the proofs of the theorems.

The next lemma will verify the conditions of more general theorems in Section \ref{section:concentration}, from which the results for the walker are immediate. 
\begin{lemma}\label{lemma:walker-oL-estimate}
Assume Assumptions 1b and 2. For any vector $v\in\bR^d$ with $\norm{v}=1$ we consider the function $f_v : \Omega \times \bZ^d \to \bR, f_v(\eta,x) = <x,v>$. Then
\[ \oL\left(S_t f_v- S_t f_v(\eta,x)\right)^k(\eta,x) \leq 2^k (R^{EP})C_b^2 C_a^{k-2}\normb{\alpha}_1^k + 2^k\norm{\alpha}_k^k. \]
\end{lemma}
\begin{proof}
We first note that $f_v$ is in the domain of the generator $L$, with
\[ Lf(\eta,x) = \suml_{z\in\bZ^d}\alpha(\theta_{-x}\eta,z)<z,v>. \]
Since this is uniformly bounded in $\eta$ and $x$ the condition of Proposition \ref{prop:exponential-point} is satisfied. Also observe that
\[ S_t f_v(\eta,x) = \int_0^t S_r^{EP}\underbrace{\left( \suml_{z\in\bZ^d}\alpha(\cdot,z)<z,v>\right)}_{=:g(\cdot)}(\theta_{-x}\eta)\,dr + f_v(\eta,x). \]
With this in mind, using Lemma \ref{lemma:oL-trick}
\begin{align*}
&\oL (S_t f_v -S_t f_v(\eta,x))^k(\eta,x) \\
&\quad= \oL \left( \int_0^r S_r^{EP}g(\theta_{-\cdot}\cdot) - S_r^{EP}g(\theta_{-x}\eta)\,dr + f_v-f_v(\eta,x)\right)^k(\eta,x)	\\
&\quad\leq 2^k \oL \left( \int_0^r S_r^{EP}g(\theta_{-\cdot}\cdot) - S_r^{EP}g(\theta_{-x}\eta)\,dr\right)^k(\eta,x) + 2^k \oL \left( f_v-f_v(\eta,x)\right)^k(\eta,x).
\end{align*}
To estimate the right hand term, as $f_v$ is independent of $\eta$, 
\begin{align*}
 \oL\left(f_v-f_v(\eta,x)\right)^k(\eta,x) &= \suml_{z\in\bZ^d}\alpha(\theta_{-x}\eta,z)<z,v>^k \\
 &\leq \suml_{z\in\bZ^d} \supl_{\eta \in \Omega}\alpha(\eta,z) \norm{z}^k = \norm{\alpha}_k^k.
\end{align*}
To treat the left hand term, we first notice that it can be rewritten in terms of the generator of environment process:
\[ \oL \left( \int_0^r S_r^{EP}g(\theta_{-\cdot}\cdot) - S_r^{EP}g(\theta_{-x}\eta)\,dr\right)^k(\eta,x) = \oLEP \left( \int_0^r S_r^{EP}g - S_r^{EP}g(\theta_{-x}\eta)\,dr\right)^k(\theta_{-x}\eta). \]
Next, we use Theorem \ref{thm:EP-estimate},b) with the fact that $\normb{g}\leq\normb{\alpha}_1$ to obtain that 
\[ \normb{\int_0^t S_r^{EP}g\,dr}\leq C_b \normb{\alpha}_1. \]
Together with Lemma \ref{lemma:EP-A2} we therefore arrive at
\[ \oL \left( \int_0^r S_r^{EP}g(\theta_{-\cdot}\cdot) - S_r^{EP}g(\theta_{-x}\eta)\,dr\right)^2(\eta,x) \leq R^{EP}C_b^2\normb{\alpha}_1^2 .\]
Higher moments we reduce to second moments via Lemma \ref{lemma:oL-trick} utilizing Theorem \ref{thm:EP-estimate},a):
\begin{align*}
 &\abs{ \int_0^r S_r^{EP}g(\theta_{-\cdot}\cdot) - S_r^{EP}g(\theta_{-x}\eta)\,dr }^k(\eta,x) \\
 &\quad\leq C_a^{k-2}\normb{\alpha}_1^{k-2}\left( \int_0^r S_r^{EP}g(\theta_{-\cdot}\cdot) - S_r^{EP}g(\theta_{-x}\eta)\,dr\right)^2(\eta,x) .
 \end{align*}
\qed
\end{proof}

\begin{proof}[Proof of Theorem \ref{thm:concentration-walker-exp}]
Let $e_i$ be an orthonormal basis of $\bR^d$, and write $e_{i+d}:=-e_i$. Then
\begin{align*}
&\bP_{\nu_1,0}\left( \norm{X_T- \bE_{\nu_2,0}X_T}> r+2dC_a\normb{\alpha}_1\right) \\
&\quad\leq \suml_{i=1}^{2d} \bP_{\nu_1,0}\left( f_{e_i} (X_T) - \bE_{\nu_2,0} f_{e_i}(X_T) > \frac{1}{2d} r+C_a\normb{\alpha}_1\right). 	
\end{align*}
 
Those probabilities we can estimate with Corollary \ref{corollary:exponential-point}, whose condition is satisfied according to Lemma \ref{lemma:walker-oL-estimate} and the conditions of this theorem. At this point we use the fact that 
\[ S_Tf_{e_i}(\eta,0)-S_Tf_{e_i}(\xi,0) = \int_0^T S^{EP}_t g_i(\eta) - S^{EP}_t g_i(\xi)\,dt \]
with $g_i(\eta) = \suml_{z\in\bZ^d} \alpha(\eta,z)f_{e_i}(z)$ to get the estimate $C_a\normb{\alpha}_1$ for the constant $a$ in Corollary \ref{corollary:exponential-point}. The constants $c_1,c_2$ are chosen appropriately.
\qed
\end{proof}

\begin{proof}[Proof of Theorem \ref{thm:concentration-walker-moment}]
We will only give a rough sketch of the proof.
It is an application of Theorem \ref{thm:moment-estimate} to obtain an estimate on the $p$th moment plus Markov's inequality. 
To use the estimate in Theorem \ref{thm:moment-estimate},
the general idea is to use the fact that 
\[ \bE_x X_t = \int_0^t S^{EP}_{s}[\suml{z\in\bZ^d}z\alpha(\cdot,z)](x) + x, \]
together with applications of Theorem \ref{thm:EP-estimate}. 
The first term on the right hand side of the estimate in Theorem \ref{thm:moment-estimate} is the main contributor and can be estimated by
\[ T^{\frac p2}2^{\frac p2}\left[(C_b^p\normb{\alpha}_1^p(R^E + \norm{L^J}_{\normb{\cdot}\to\normb{\cdot}})^p +\norm{\alpha}_2^p\right] . \]
The second term is estimated by
\[ 2^p\left(C_b^p\normb{\alpha}_1^p+\norm{\alpha}_p^p\right) \]
and the third by
\[ (C_a\normb{\alpha}_1)^p .\]
Combining various constants, we then get the estimate
\[ \bE_{\nu_1}\norm{X_T-\bE_{\nu_2}X_T}^p \leq c_p C^p(T^{\frac p2}+1). \]\qed
\end{proof}

\subsection{Transience and recurrence}
Using the concentration estimates and the convergence to Brownian motion, questions like transience or recurrence of the walker can be treated rather easily. The following two statements about transience and recurrence are not meant to be exhaustive, instead they showcase how typical use the stronger statements of the previous sections.
\begin{theorem}[Transience]
Assume Assumptions 1b and 2 and suppose $\norm{\alpha}_{2+\epsilon}<\infty$ for some $\epsilon>0$. 
If asymptotic speed $v=\lim_{t\to\infty}\frac{X_t}{t}$ is non-zero, then the random walk $X_t$ is transient.
\end{theorem}
\begin{proof}
Since $a:=\sum_{z\in\bZ^d}\sup_{\eta\in \Omega} \alpha(\eta,z)<\infty$, the duration the walker $X_t$ stays at the origin at each visit is at least $exp(a)$-distributed. Hence it is sufficient for transience to show that $\bP_{0,\eta}(X_t=0)$ is integrable in $t$, which can be seen from Theorem \ref{thm:concentration-walker-moment}:
\begin{align*}
\bP_{\eta,0}(X_t=0) &\leq \bP_{\eta,0}(\abs{X_t-vt}\geq vt) 
\leq \left(c_{2+\epsilon}c^{2+\epsilon}\frac{t^{\frac{2+\epsilon}{2}}+1}{(vt)^{2+\epsilon}}\right)\wedge 1,
\end{align*}
where we used the fact that 
\[ vt = \bE_{\mu^{EP},0}X_t. \]
\qed
\end{proof}
\begin{theorem}[Recurrence]
Let the dimension $d=1$ and assume Assumptions 1b and 2. Suppose the walker only has jumps of size 1, i.e.
\[ \alpha(\eta,z)= 0 \quad \forall\,\eta\in E^{bZ^d},z\neq \{-1,1\}, \]
and that it has 0 speed.
Then $X_t$ is recurrent.
\end{theorem}
\begin{proof}
By Theorem \ref{thm:CLT}, $T^{-\frac12}X_{tT}$ converges to Brownian motion. Hence there exists(with probability 1) an infinite sequence $t_1<t_2<\dots$ of times with $X_{t_{2n}}<0$ and $X_{t_{2n+1}}>0$, $n\in\bN$. As the walker has only jumps of size 1, it will traverse the origin between $t_n,t_{n+1}$ for any $n\in \bN$. \qed
\end{proof}

\section{Appendix: General concentration results}\label{section:concentration}
In this section we will prove general concentration results for Markov processes. Hence we will forget the
connection to random walks in dynamic random environments. To avoid a clash of notation, we will assume that
$F$ is a polish space, elements of it are denoted $x,y$, $(Y_t)_{t\geq0}$ a Feller process on $F$, with generator $A$ and semigroup $(P_t)_{t\geq0}$ (on an appropriate space like $\cC_0(F), \cC_b(F), \cC(F), \cB(F),...$). The canonical filtration is denoted by $(\fF_t)_{t\geq0}$.

The method of the proofs will use certain martingale approximation.
\begin{notation}\label{notation:martingale}
Given $f:F\to \bR$, we define the martingale
\begin{align*}
 M_T^f(t) &:= \erwc{f(Y_T)}{\fF_t} - \erwc{f(Y_T)}{\fF_0}, \quad\quad t\in[0,T], \\
&= P_{T-t}f(Y_t) - P_T f (Y_0).
\end{align*}
\end{notation}
\begin{remark}
Note that this martingale is different from the canonical martingale $f(Y_t)-\int_0^t A f(Y_s)\,ds$, but there are some similarities.
\end{remark}
To avoid questions regarding the domain of the generator $A$, we also introduce (non-linear) versions of $A$ in a natural way which are defined for all functions.
\begin{notation}
For any $f: F \to \bR$, define
\[ \oA f(x) := \limsup_{\epsilon\to0}\frac1\epsilon (P_\epsilon f(x) - f(x)); \]
\[ \uA f(x) := \liminf_{\epsilon\to0}\frac1\epsilon (P_\epsilon f(x) - f(x)). \]
We call $\oA$ upper generator and $\uA$ lower generator.
\end{notation}
\begin{remark}\mbox{}
\par{a)}
Note that $\oA f(x), \uA f(x) \in \bR \cup \{\pm\infty\}$. But if $f \in dom(A)$, then
$\oA f = \uA f = A f$.
\par{b)} If we look at the martingale 
\begin{align}\label{eq:classical-martingale}
\widetilde{M}^f(t) = f(Y_t)-\int_0^t Af(Y_s)\,ds, 	
\end{align}
then its predictable quadratic variation process is given by
\begin{align}
\left< \widetilde{M}^f \right>_t 
&= \int_0^t A f^2(Y_s) - 2f(Y_s)Af(Y_s)\,ds \nonumber\\
&= \int_0^t \left[A \left(f-f(Y_s)\right)^2\right](Y_s)\,ds \label{eq:classical-predqv}
\end{align}
under the assumption that $f,f^2\in dom(A)$. If that is not the case, then we still have an upper (lower) bound on $\left<\widetilde{M}^f\right>$ when using the upper generator $\oA$ (lower generator $\uA$) in \eqref{eq:classical-predqv}.
\end{remark}
The following lemma shows that $[\oA(f-f(x))(g-g(x))](x)]$ still has the property of a quadratic form.
\begin{lemma}\label{lemma:chauchy-schwarz}
Let, $f,g : F\to\bR$. Then
\[ \abs{\oA[(f-f(x))(g-g(x))](x)}^2 \leq \oA(f-f(x))^2(x)\oA(g-g(x))^2(x). \]
\end{lemma}
\begin{proof}[Proof of the lemma]
Utilizing the definition of $\oA$ and Chauchy-Schwarz inequality,
\begin{align*}
&\abs{\oA[(f-f(x))(g-g(x))](x)}^2 \\
&\quad\leq \limsup_{\epsilon\to0}\frac1{\epsilon^2}\abs{\bE_x(f(Y_\epsilon)-f(x))(g(Y_\epsilon)-g(x))}^2\\
&\quad\leq \limsup_{\epsilon\to0}\frac1{\epsilon^2}\left[\bE_x(f(Y_\epsilon)-f(x))^2\bE_x(g(Y_\epsilon)-g(x))^2\right] \\
&\quad\leq \left(\oA(f-f(x))^2(x)\right)\left(\oA(g-g(x))^2(x)\right)\\
\end{align*}
\qed
\end{proof}
\begin{remark}
It is also interesting to note that
\[ \left<f,g\right>_x = A[(f-f(x))(g-g(x))](x) \]
is a positive semi-definite bilinear form. This form can be related to the martingales \eqref{eq:classical-martingale} by their predictable covariation process, which is given by
\[ \left< \widetilde{M}^f, \widetilde{M}^g \right>_t = \int_0^t \left<f,g\right>_{Y_s}\,ds. \]
\end{remark}
Before going into details about the martingale $M^f_T$ and its properties we need the following little fact.
\begin{lemma}\label{lemma:little-helper}
Fix $\delta>0$, $k\in\bN,k\geq2$ and a probability measure $\mu$. Let $f,g : [0,\delta] \to L^k(\mu)$. Assume that
\[ \lim_{\epsilon\to0} \frac1\epsilon \mu\left( f(\epsilon)^k \right) = a\in\bR,\quad \lim_{\epsilon\to0} \frac1\epsilon \mu\left( g(\epsilon)^k \right) = 0 .\]
Then
\[ \lim_{\epsilon\to0} \frac1\epsilon \mu\left( (f(\epsilon)+g(\epsilon))^k\right) = a. \]
The same holds true for $\limsup$ and $\liminf$ as long as $\lim_{\epsilon\to0} \frac1\epsilon \mu\left( g(\epsilon)^k \right) = 0$.
\end{lemma}
\begin{proof}
First,
\begin{align*}
\mu\left( (f(\epsilon)+g(\epsilon))^k \right) &= \suml_{j=0}^k \binom{k}{j} \mu\left( f(\epsilon)^jg(\epsilon)^{k-j}\right).
\end{align*}
By H\"older's inequality with $p=\frac{k}{j},q=\frac{k}{k-j}$,
\[ \mu \left( \abs{f(\epsilon)^jg(\epsilon)^{k-j}} \right) \leq \left(\mu\left( f(\epsilon)^k\right) \right)^{\frac{j}{k}} \left(\mu\left( g(\epsilon)^k\right)\right)^{\frac{k-j}{k}}. \]
Therewith and with the assumptions on the limits,
\[ \lim_{\epsilon\to0}\frac{1}{\epsilon} \mu\left( \abs{f(\epsilon)^jg(\epsilon)^{k-j}}\right) \leq \left(\lim_{\epsilon\to0}\frac{1}{\epsilon}\mu\left( f(\epsilon)^k\right)^{\frac{j}{k}}\right) \left(\lim_{\epsilon\to0}\frac{1}{\epsilon}\mu\left( g(\epsilon)^k\right)^{\frac{k-j}{k}}\right) = 0 \]
for any $j<k$.
\qed
\end{proof}

Now we will state and prove the key lemma regarding the martingale \ref{notation:martingale}. It is very parallel to Lemma 2.5 in \cite{REDIG:VOLLERING:10}, where the identical approach is used to deal with the analogue martingale constructed from $\int_0^T f(Y_t)\,dt$ instead of $f(Y_T)$.
\begin{lemma}\label{lemma:point-key-lemma}
Fix $k\in \bN, k\geq 2$ If
\[ \lim_{\epsilon\to0} \frac{1}{\epsilon} \bE_x \left(P_{T-t-\epsilon}f(Y_\epsilon)-P_{T-t}f(Y_\epsilon)\right)^k=0 \quad \forall\, x\in F,\] 
then:
\begin{enumerate}
\item $\limsup\limits_{\epsilon\to0} \frac{1}{\epsilon}\erwc{(M_T^f(t+\epsilon)-M_T^f(t))^k}{\fF_t} = \oA\left(P_{T-t} f(\cdot) - P_{T-t}f(Y_t)\right)^k(Y_t); $
\item $\liminf\limits_{\epsilon\to0} \frac{1}{\epsilon}\erwc{(M_T^f(t+\epsilon)-M_T^f(t))^k}{\fF_t} = \uA\left(P_{T-t} f(\cdot) - P_{T-t}f(Y_t)\right)^k(Y_t). $
\end{enumerate}
\end{lemma}
\begin{remark}
If $\sup_{x\in F} \oA f(x)<\infty$, then the condition is satisfied for any $k\geq2$.
\end{remark}
\begin{proof}
First,
\begin{align*}
M^f_T(t+\epsilon)-M^f_T(t) &= P_{T-t-\epsilon}f(Y_{t+\epsilon})-P_T f(Y_0) - P_{T-t}f(Y_t)+P_T f(Y_0)\\
&= P_{T-t-\epsilon}f(Y_{t+\epsilon})-P_{T-t}f(Y_{t+\epsilon}) + P_{T-t}f(Y_{t+\epsilon}) -P_{T-t}f(Y_{t}).
\end{align*}
By our assumptions and the Markov property, 
\[\lim_{\epsilon\to0} \frac1\epsilon\erwc{(P_{T-t-\epsilon}f(Y_{t+\epsilon})-P_{T-t}f(Y_{t+\epsilon}))^k}{\fF_t} = 0.\]
Hence, by Lemma \ref{lemma:little-helper}
\begin{align*}
&\limsup_{\epsilon\to0}\frac1\epsilon\erwc{(M^f_T(t+\epsilon)-M^f_T(t))^k}{\fF_t} \\
&\quad= \limsup_{\epsilon\to0}\frac1\epsilon\erwc{(P_{T-t}f(Y_{t+\epsilon}) -P_{T-t}f(Y_{t}))^k}{\fF_t}	\\
&\quad= \oA\left(P_{T-t} f(\cdot) - P_{T-t}f(Y_t)\right)^k(Y_t).
\end{align*}
The argument for $\liminf$ is analogue.
\qed
\end{proof}
\begin{proposition}\label{prop:point-quadratic-variation}
Assume that for all $0< t \leq T$ and all $x\in F$ $(P_t f -P_t f(x))^2 \in dom(A)$ and 
\[ \lim_{\epsilon\to0} \frac{1}{\epsilon} \bE_x \left( P_{t-\epsilon}f(Y_\epsilon) - P_{t}f(Y_\epsilon) \right)^2 =0.\] 
Then the predictable quadratic variation of $M_T^f$ is
\begin{align*}
\left< M_T^f \right>_t &= \int_0^t A\left(P_{T-s} f(\cdot) - P_{T-s}f(Y_s)\right)^2(Y_s)\, ds	.
\end{align*}
\end{proposition}
\begin{proof}
By Lemma \ref{lemma:point-key-lemma}, 
\[ \frac{d}{dt}\left<M^f_T\right>_t = \lim_{\epsilon\downarrow0} \frac1\epsilon \erwc{\left(M^f_T(t+\epsilon)-M^f_T(t)\right)^2}{\fF_t} = A\left(P_{T-t} f(\cdot) - P_{T-t}f(Y_t)\right)^2(Y_t). \]
\qed
\end{proof}

\begin{theorem}\label{thm:point-exponential-martingale}
Assume that for all $k\in\bN, k\geq2$, all $x\in F$ and all $0< t \leq T$,
\[ \lim_{\epsilon\to0} \frac{1}{\epsilon} \bE_x \left( P_{t-\epsilon}f(Y_\epsilon) - P_{t}f(Y_\epsilon) \right)^k =0.\] 
Also assume that 
\[ \supl_{0\leq t \leq T} \supl_{x\in F}\suml_{k=2}^\infty\frac{1}{k!}\overline{A}\left(P_t f - P_t f(x)\right)^k(x)<\infty. \]
Write
\[ \overline{N}^f_T(t):=\exp\left({M^f_T(t)- \int_0^t \suml_{k=2}^\infty \frac{1}{k!} \oA\left(P_{T-s} f - P_{T-s} f(Y_s)\right)^k(Y_s)\,ds}\right),\quad 0\leq t\leq T; \]
\[ \underline{N}^f_T(t):=\exp\left({M^f_T(t)- \int_0^t \suml_{k=2}^\infty \frac{1}{k!} \uA\left(P_{T-s} f - P_{T-s} f(Y_s)\right)^k(Y_s)\,ds}\right),\quad 0\leq t\leq T. \]
Then $(\overline{N}_T(t))_{0\leq t\leq T}$ is a supermartingale and $(\underline{N}_T^f(t))_{0\leq t\leq T}$ is a submartingale. If  $(P_{T-t}f-P_{T-t}f(x))^k(x)\in dom(A)$ for all $x$ and $k$, then both are equal and a martingale.
\end{theorem}
\begin{proof}
It is sufficient to prove that $\limsup_{\epsilon\to0}\frac1\epsilon\erwc{\overline{N}^f_T(t+\epsilon)-\overline{N}_T^f(t)}{\fF_t}\leq 0$ (or analogue with $\liminf$ for $\underline{N}_T^f$). For a shorter notation, write
\[ \psi(s,x,k):= \oA\left(P_{T-s} f - P_{T-s} f(x)\right)^k(x).\]
Then
\begin{align*}
&\limsup_{\epsilon\to0}\frac1\epsilon\erwc{\overline{N}^f_T(t+\epsilon)-\overline{N}_T^f(t)}{\fF_t}\\
&\quad=\overline{N}_T^f(t)\limsup_{\epsilon\to0}\frac1\epsilon\erwc{\exp\left({M^f_T(t+\epsilon)-M^f_T(t)- \int_t^{t+\epsilon} \suml_{k=2}^\infty \frac{1}{k!} \psi(s,Y_s,k)\,ds}\right)-1}{\fF_t} \\
&\quad\leq\overline{N}_T^f(t)\suml_{l=1}^\infty\frac{1}{l!} \limsup_{\epsilon\to0}\frac1\epsilon\erwc{\left({M^f_T(t+\epsilon)-M^f_T(t)- \int_t^{t+\epsilon} \suml_{k=2}^\infty \frac{1}{k!} \psi(s,Y_s,k)\,ds}\right)^l}{\fF_t}.
\end{align*}
First, since $\overline{N}^f_T(t)>0$, we can ignore it. Next, we study the terms individually in $l$. For $l=1$, as $M^f_T$ is a martingale, we get just the term
\[ -\suml_{k=2}^\infty \frac{1}{k!} \psi(t,Y_t,k). \]
For $l>1$, we use the fact that 
\[ \int_t^{t+\epsilon} \suml_{k=2}^\infty \frac{1}{k!} \psi(s,Y_s,k)\,ds \leq \epsilon \supl_{0\leq t \leq T}  \sup_{x\in F} \suml_{k=2}^\infty \frac{1}{k!} \psi(t,x,k) = \epsilon\cdot const. \]
Therefore 
\[ \lim_{\epsilon\to0} \frac{1}{\epsilon}\erwc{\left(\int_t^{t+\epsilon} \suml_{k=2}^\infty \frac{1}{k!} \psi(s,Y_s,k)\,ds\right)^l}{\fF_t} = 0, \]
and by Lemma \ref{lemma:little-helper}, we can drop that term. Finally, by Lemma \ref{lemma:point-key-lemma},
\[ \limsup_{\epsilon\to0}\frac1\epsilon\erwc{\left(M^f_T(t+\epsilon)-M^f_T(t)\right)^l}{\fF_t} = \psi(s,Y_t,l) .\]
So the terms for $l\geq2$ sum to $\suml_{l=2}^\infty \frac{1}{l!} \psi(t,Y_t,l)$, which cancels exactly with the term for $l=1$. \qed
\end{proof}
As an direct consequence we can estimate the exponential moment of $f(Y_T)$.
\begin{proposition}\label{prop:exponential-point}
Assume that for all $k\in\bN, k\geq2$, all $x\in F$ and all $0< t \leq T$,
\[ \lim_{\epsilon\to0} \frac{1}{\epsilon} \bE_x \left( P_{t-\epsilon}f(Y_\epsilon) - P_{t}f(Y_\epsilon) \right)^k =0.\] 
Let $\nu_1,\nu_2$ be two arbitrary probability measures on $F$. Then
\[ \bE_{\nu_1} e^{f(X_T)-\nu_2(P_T f)} \leq c(\nu_1,\nu_2) \exp\left({\int_0^T \supl_{x\in F} \suml_{k=2}^\infty\frac{1}{k!}\oA\left(P_t f - P_t f(x)\right)^k(x)\,dt}\right), \]
and 
\[ \bE_{\nu_1} e^{f(X_T)-\nu_2(P_T f)} \geq c(\nu_1,\nu_2) \exp\left({\int_0^T \infl_{x\in F} \suml_{k=2}^\infty\frac{1}{k!}\uA\left(P_t f - P_t f(x)\right)^k(x)\,dt}\right), \]
with 
\[ c(\nu_1,\nu_2) = \int e^{P_T f(x) - \nu_2 (P_T f)}\nu_1(dx). \]
Notably, $c(\delta_x,\delta_x)=1$.
\end{proposition}
\begin{proof}
As a first observation,
\[ \bE_{\nu_1}\left( e^{f(X_T)-\nu_2(P_T f)}\right) = \bE_{\nu_1}\left(e^{P_T f(Y_0) - \nu_2 (P_T f)}\erwc{e^{M_T^f(T)}}{\fF_0} \right). \]
Next, we rewrite the conditional expectation as
\[ \erwc{e^{M_T^f(T)}}{\fF_0} = \erwc{\overline{N}_T^f(T)e^{\int_0^T\suml_{k=2}^\infty\frac{1}{k!}\overline{A}\left(P_{T-t} f - P_{T-t} f(Y_t)\right)^k(Y_t)\,dt}}{\fF_0}, \]
where $\overline{N}_T^f$ is the supermartingale from Theorem \ref{thm:point-exponential-martingale}. As $\overline{N}_T^f$ is positive, we can estimate that term by taking the supremum over the possible values of $Y_t$ in the integral, so that we obtain
\[ \erwc{e^{M_T^f(T)}}{\fF_0} \leq \erwc{\overline{N}_T^f(T)}{\fF_0}e^{\int_0^T\sup_{x\in F}\suml_{k=2}^\infty\frac{1}{k!}\overline{A}\left(P_{T-t} f - P_{T-t} f(x)\right)^k(x)\,dt} .\]
As $\overline{N}_T^f$ is a supermartingale and $\overline{N}^f_T(0) = 1$, we finally arrive at
\[ \bE_{\nu_1} e^{f(X_T)-\nu_2(P_T f)} \leq c(\nu_1,\nu_2) \exp\left({\int_0^T \supl_{x\in F} \suml_{k=2}^\infty\frac{1}{k!}\oA\left(P_{T-t} f - P_{T-t} f(x)\right)^k(x)\,dt}\right). \]
Reversing the direction of time yields the claim, and the proof for the other bound is analogue.
\qed
\end{proof}

\begin{corollary}\label{corollary:exponential-point}
Let $f: F \to \bR$ be in $dom(L)$ with $\norm{L f}_\infty < \infty$.
Assume that 
\[ \oA(P_t f -P_t f(x))^k(x) \leq c_1 c_2^k \]
uniformly in $x \in F$ and $t \in [0,T]$. Then, for any $r>0$,
\[ \bP_x\left( f(Y_T) - \bE_x f(Y_T) > r \right) \leq e^{-\frac{\frac12 \left(\frac{r}{c_2}\right)^2}{Tc_1 + \frac13 \frac{r}{c_2}}}. \]
If $f$ is bounded, then for any probability measures $\nu_1,\nu_2$ on $F$, 
\[ \bP_{\nu_1}\left( f(Y_T) - \bE_{\nu_2} f(Y_T) > r + \norm{P_T f}_{osc} \right) \leq  e^{-\frac{\frac12 \left(\frac{r}{c_2}\right)^2}{Tc_1 + \frac13 \frac{r}{c_2}}}. \]
\end{corollary}
\begin{proof}
By Markov's inequality,
\begin{align*}
\bP_{\nu_1}(f(Y_T)-\bE_{\nu_2} f(Y_T) > r ) 
&\leq \bE_{\nu_1} e^{\lambda f(Y_T)-\bE_{\nu_2} \lambda f(Y_T)} e^{-\lambda r} \\
&\leq e^{T c_1 \sum_{k=2}^\infty \lambda^k c_2^k-\lambda r + \lambda\norm{P_T f}_{osc}}, 
\end{align*}
where the last line is the result from Proposition \ref{prop:exponential-point}. Note that $c(\nu_1,\nu_2)\leq e^{\lambda\norm{P_T f}_{osc}}$ and $c(\delta_x,\delta_x)=e^0$ For now we continue with the term $\lambda\norm{P_T f}_{osc}$ present. Through optimizing $\lambda$, the exponent becomes
\[ \frac{r-\norm{P_T f}_{osc}}{c_2}-\left(T c_1+\frac{r-\norm{P_T f}_{osc}}{c_2}\right)\log\left(\frac{r-\norm{P_T f}_{osc}}{T c_1 c_2}+1\right). \]
By writing $r'=r-\norm{P_T f}_{osc}$, we continue with 
\[ \frac{r'}{c_2}-\left(T c_1+\frac{r'}{c_2}\right)\log\left(\frac{r'}{T c_1 c_2}+1\right). \]
This step is not necessary in the case of a fixed initial configuration $x$, as then $r'=r$.
To show that his term is less than $\frac{-\frac12(\frac{r'}{c_2})^2}{Tc_1+\frac13\frac{r'}{c_2}}$, we first rewrite it as the following inequality:
\[ \log\left(\frac{r'}{T c_1 c_2}+1\right) \geq \frac{\frac{\frac12(\frac{r'}{c_2})^2}{Tc_1+\frac13\frac{r'}{c_2}}+\frac{r'}{c_2}}{T c_1+\frac{r'}{c_2}}. \]
Through comparing the derivatives, one concludes that the left hand side is indeed bigger than the right hand side. \qed
\end{proof}

The next lemma is a slightly refined version of Corollary 2.7 in \cite{REDIG:VOLLERING:10}.
\begin{lemma}\label{lemma:exponential-concentration}
Fix $T>0$ and $f:Y\to\bR$. Assume that for some $T'>T$ and all $k\in\bN, k\geq2$, $\bE_x sup_{0\leq t \leq T'} f(Y_t)^k < \infty.$
Assume also that for some constants $c_1, c_2>0$ the following estimates hold:
\[ \sup_{x,y\in X} \sup_{0\leq T' \leq T} \int_0^{T'} P_t f(x)- P_t f(y)\,dt \leq c_1;\]
\[ \sup_{x \in X} \sup_{0\leq T' \leq T} \overline{A}\left(\int_0^{T'} P_t f- P_t f(x)\,dt\right)^2(x) \leq c_2 c_1^2. \]
Then
\[ \bP_{\nu_1}\left( \int_0^T f(X_t)\,dt > c_1(r+1) + \int_0^T \nu_2(P_t f)\,dt \right) \leq e^{- \frac{\frac12r^2}{Tc_2 + \frac13r}} \]
for any $r>0$ and any probability measures $\nu_1,\nu_2$ on $F$. If $\nu_1,\nu_2$ are equal and a Dirac measure, then one can replace $(r+1)$ by $r$.
\end{lemma}
The proof of this lemma is analogue to the proof of Corollary \ref{corollary:exponential-point}, using the corresponding exponential moment bound in Theorem 2.6 in \cite{REDIG:VOLLERING:10}.
\begin{theorem}\label{thm:moment-estimate} Let $f: F \to\bR$ and fix $T>0$.
Assume that for all $x\in F$ and all $0<t<T$
\[ \lim_{\epsilon\to0} \frac{1}{\epsilon} \bE_x \left(P_{t-\epsilon}f(Y_\epsilon)-P_{t}f(Y_\epsilon)\right)^2=0 .\] 
Then, for any probability measures $\nu_1, \nu_2$ on $F$ and any $p\geq 2$,
\begin{align*}
 \left(\bE_{\nu_1}\abs{ f(Y_T)-\bE_{\nu_2}f(Y_T) }^p\right)^{\frac1p} & \leq C_p \left[ \left( \bE_{\nu_1}\left(\int_0^T \oA(P_{T-t}f - P_{T-t}f(X_t))^2(X_t)\,dt\right)^{\frac p2}\right)^{\frac1p}\right. \\
&\quad \left. +\left(\bE_{\nu_1}\sup_{0< t \leq T} \abs{P_{T-t}f(Y_t)-P_{T-t}f(Y_{t-})}^p\right)^{\frac1p}\right] \\
&\quad + \left(\int \abs{P_T f(x) - \nu_2(P_T f)}^p\,\nu_1(dx) \right)^{\frac1p} .
\end{align*}
\end{theorem}
The proof is analogue to the proof of Theorem 2.9 in \cite{REDIG:VOLLERING:10}.


\bibliography{BibCollection}{}
\bibliographystyle{plain}

\end{document}